\pdfoutput=1
\documentclass[11pt]{article}

\usepackage[a4paper,margin=1in]{geometry}
\usepackage{amsfonts,amssymb}
\usepackage{amsthm}
\usepackage{graphicx}
\usepackage{algorithm}
\usepackage{mathtools}
\usepackage{booktabs}
\usepackage{array}
\usepackage{float}
\usepackage{algpseudocode}
\usepackage[utf8]{inputenc}
\usepackage[numbers,sort&compress]{natbib}
\usepackage{enumitem}
\usepackage{lmodern}
\usepackage{microtype}
\usepackage{multirow}
\usepackage{subcaption}
\usepackage[section]{placeins}
\usepackage[colorlinks=true,linkcolor=blue,citecolor=blue,urlcolor=blue]{hyperref}
\usepackage[capitalize,noabbrev]{cleveref}

\ifpdf
  \DeclareGraphicsExtensions{.pdf,.png,.jpg,.eps}
\else
  \DeclareGraphicsExtensions{.eps}
\fi

\graphicspath{{figures/exp1/}{figures/exp2/}{figures/exp3/}{figures/}}

\newtheorem{theorem}{Theorem}[section]
\newtheorem{lemma}[theorem]{Lemma}
\newtheorem{proposition}[theorem]{Proposition}
\newtheorem{corollary}[theorem]{Corollary}
\newtheorem{assumption}[theorem]{Assumption}
\crefname{assumption}{Assumption}{Assumptions}
\Crefname{assumption}{Assumption}{Assumptions}
\theoremstyle{remark}
\newtheorem{remark}[theorem]{Remark}
\theoremstyle{definition}
\crefname{remark}{Remark}{Remarks}
\Crefname{remark}{Remark}{Remarks}

\newtheorem{definition}[theorem]{Definition}
\crefname{example}{Example}{Examples}
\Crefname{example}{Example}{Examples}

\allowdisplaybreaks

\newcommand{\email}[1]{\texttt{#1}}
\newenvironment{keywords}{\par\smallskip\noindent\textbf{Keywords.} }{\par}
\newenvironment{AMS}{\par\smallskip\noindent\textbf{AMS subject classifications.} }{\par}
\newcommand{\R}{\mathbb{R}}
\newcommand{\ip}[2]{\left\langle #1,#2 \right\rangle}
\newcommand{\norm}[1]{\left\lVert #1 \right\rVert}
\newcommand{\abs}[1]{\left\lvert #1 \right\rvert}
\newcommand{\normM}[2]{\left\lVert #1 \right\rVert_{#2}}
\newcommand{\pred}{\operatorname{pred}}
\newcommand{\ared}{\operatorname{ared}}

\newcommand{\lammin}{\lambda_{\min}}
\newcommand{\lammax}{\lambda_{\max}}
\newcommand{\diag}{\operatorname{diag}}
\newcommand{\cond}{\operatorname{cond}}

\newcommand{\vol}{\operatorname{vol}}
\newcommand{\cO}{\mathcal{O}}
\newcommand{\MATRO}{\textup{MATRO}}
\newcommand{\MetSel}{\textup{\textsc{MetricSelection}}}

\newcommand{\Mclass}{\mathcal{M}}
\newcommand{\Cclass}{\mathcal{C}}
\newcommand{\Bball}{\mathcal{B}}

\newcommand{\omunit}{\omega_n}

\title{\MATRO{}: Metric-Aware Trust-Region Optimization\\with Fully Quadratic Models}

\author{
  Wei Hu\thanks{
    LSEC, ICMSEC, Academy of Mathematics and Systems Science,
    Chinese Academy of Sciences, Beijing 100190, China;
    and University of Chinese Academy of Sciences, Beijing 100049, China
    (\email{huwei@amss.ac.cn}). Corresponding author.}
  \and
  Pengcheng Xie\thanks{
    Lawrence Berkeley National Laboratory,
    Berkeley, CA 94720, USA
    (\email{pxie98@gmail.com}).}
  \and
  Ya-Xiang Yuan\thanks{
    LSEC, ICMSEC, Academy of Mathematics and Systems Science,
    Chinese Academy of Sciences, Beijing 100190, China
    (\email{yyx@lsec.cc.ac.cn}).}
  \and
  Li Zhang\thanks{
    LSEC, ICMSEC, Academy of Mathematics and Systems Science,
    Chinese Academy of Sciences, Beijing 100190, China
    (\email{zhangli2022@lsec.cc.ac.cn}).}
}

\ifpdf
\hypersetup{
  pdftitle={MATRO: Metric-Aware Trust-Region Optimization with Fully Quadratic Models},
  pdfauthor={W. Hu, P. Xie, Y. Yuan, and L. Zhang}
}
\fi

\begin{document}
\maketitle

\begin{abstract}
Model-based derivative-free trust-region methods build local interpolation models and
restrict trial steps to regions where those models are reliable.  This paper studies
the shape of that region.  When an objective is poorly scaled or locally anisotropic,
a Euclidean ball can be governed by the steepest local direction and can restrict
progress along directions of slow variation.  We propose \MATRO{} (Metric-Aware
Trust-Region Optimization), a fully quadratic interpolation framework in which the
trust region is the ellipsoid \(s^\top M_k s\le \Delta_k^2\).  For any positive
definite metric \(M_k\), the induced variable \(y=M_k^{1/2}s\) converts the
ellipsoidal subproblem into a standard Euclidean trust-region subproblem, so model
decrease, ratio tests, radius updates, poisedness, and fully quadratic error bounds
can be stated in induced coordinates under a uniform metric contract.  The metric is
selected from the interpolation Hessian: positive definite quadratics yield a unique
volume-normalized curvature metric that isotropizes the induced Hessian and gives a
truncated Newton step, while indefinite fitted Hessians motivate an absolute-curvature
metric that balances curvature magnitudes without changing curvature signs.  Under the
standard fully quadratic assumptions and the metric contract, \MATRO{} retains the
first-order evaluation-complexity order \(\cO(n^2\varepsilon^{-2})\).  Experiments on
Mor\'e--Wild benchmarks, controlled anisotropy tests, and two-dimensional trajectories
show that curvature-shaped regions are most effective when the interpolation Hessian
captures stable local anisotropy, while dense linear algebra is most visible at loose
accuracies or on inexpensive analytic tests.
\end{abstract}



\section{Introduction}
\label{sec:intro}

Derivative-free optimization (DFO) addresses problems in which function values are
available but reliable derivatives are not.  Such problems arise when the objective is
produced by a simulation, experiment, or legacy code, and a single evaluation may be
far more expensive than the algebra performed by the optimizer.  Model-based
trust-region methods form one of the main deterministic approaches in this setting:
they construct a local surrogate from sampled values, minimize the surrogate in a
neighborhood of the current point, and accept the trial point according to the agreement
between actual and predicted reduction \cite{ConnScheinbergVicente2009,LarsonMenickellyWild2019}.

The standard analysis of interpolation-based DFO is built around two ingredients: a
local model that is accurate enough near the current point and a sampling set with
adequate geometry.  Fully linear and fully quadratic models, together with poised
interpolation sets, make this principle precise.  The present paper studies a related
geometric choice that is often fixed before the analysis begins: the shape of the region
where the model is trusted.  Most interpolation-based trust-region solvers use a
Euclidean ball in the original variables.  This convention is stable and simple to manage,
but it represents all directions by a single scalar radius.

A single scalar radius can be restrictive in poorly scaled or locally anisotropic regions.
A narrow curved valley gives the basic example.  Across the valley the objective changes
rapidly, so a reliable model may require short transverse steps.  Along the valley floor
the objective changes more slowly, and longer steps may be beneficial.  A
spherical trust region must compromise between these two scales.  If the radius is
chosen for the transverse direction, movement along the valley floor may be needlessly
slow; if it is chosen for the tangential direction, the trial step may leave the region in
which the model is reliable across the valley.  This is not simply a failure of variable
preprocessing.  The relevant directions can rotate during the run, and in a black-box
setting they must be inferred from the function values already collected.

\begin{figure}[t]
    \centering
    \includegraphics[width=.45\textwidth]{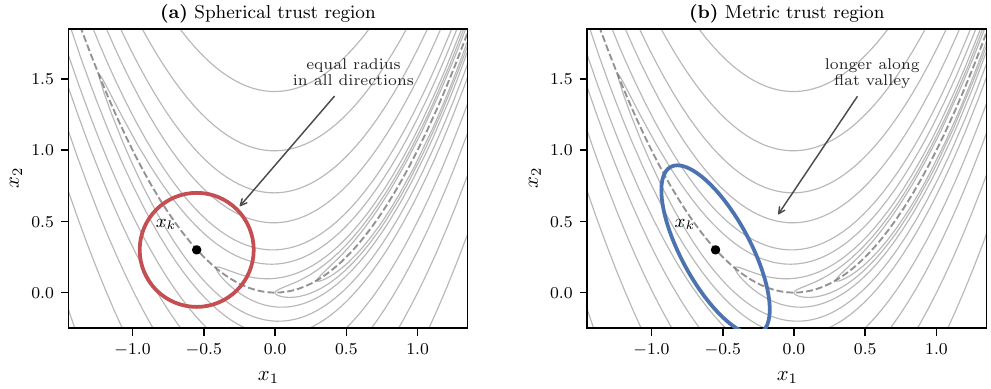}
    \caption{Spherical and metric trust regions on Rosenbrock-type contours.  A ball
    in the original variables uses the same length scale in every direction, whereas an
    ellipsoid can be aligned with a local valley.  The figure is illustrative: \MATRO{}
    constructs the metric from interpolation curvature, not from exact derivative
    information.}
    \label{fig:motivation}
\end{figure}

Changing the shape of the trust region is equivalent to changing the metric in the
trust-region subproblem.  For a symmetric positive definite matrix \(M_k\), define
\begin{equation}
    \Bball_{M_k}(x_k,\Delta_k)
    :=\{x_k+s:s^\top M_k s\le \Delta_k^2\}.
    \label{eq:intro-M-ball}
\end{equation}
The scalar \(\Delta_k\) controls the size of the region, while \(M_k\) controls its shape
and orientation.  With the induced variable
\begin{equation}
    y=M_k^{1/2}s,
    \label{eq:intro-induced}
\end{equation}
the ellipsoid \eqref{eq:intro-M-ball} becomes the Euclidean ball \(\norm{y}\le
\Delta_k\).  Thus an ellipsoidal subproblem in the original variables is an ordinary
Euclidean trust-region subproblem after a linear change of coordinates.  This simple
observation is the first organizing principle of the paper: the trust-region mechanism
can be analyzed for an arbitrary bounded metric before any particular metric-selection
rule is specified.

The second organizing principle is that a fully quadratic DFO method already computes
a Hessian matrix.  Classical full quadratic solvers use this matrix in the quadratic
objective but usually leave the trusted region spherical in the original variables.  We ask
whether the same fitted curvature can also define the shape of the region over which the
model is trusted.  This question is meaningful only if the metric is stable: the fitted
Hessian may be inaccurate in early iterations, nearly singular, or indefinite.  The role of
the metric is therefore not to convexify the model.  The signed Hessian remains in the
trust-region subproblem, and the metric supplies a positive definite length scale for the
step.

The ideal positive definite case explains the construction.  For the quadratic
\begin{equation}
    f(x)=\frac12 x^\top Hx,
    \qquad H=\diag(1,\kappa),
    \qquad \kappa\gg1,
    \label{eq:intro-two-d-quad}
\end{equation}
the Euclidean metric leaves the induced Hessian with condition number \(\kappa\).  The
volume-normalized curvature metric
\begin{equation}
    M_\star=\frac{H}{\det(H)^{1/2}}
    =\diag(\kappa^{-1/2},\kappa^{1/2})
    \label{eq:intro-Mstar}
\end{equation}
gives
\begin{equation}
    M_\star^{-1/2}HM_\star^{-1/2}=\kappa^{1/2}I.
    \label{eq:intro-induced-isotropic}
\end{equation}
Thus the quadratic is isotropic in the variables used by the subproblem.  In the
original variables, the trust region expands along the flat direction and contracts along
the steep direction.  If a condition-number cap \(\cond(M)\le K\) is imposed for
stability, the best commuting metric in this two-dimensional example can remove only a
factor \(K\) of the original condition number.  \Cref{tab:intro-two-d} summarizes this
calculation.

\begin{table}[t]
    \centering
    \caption{The quadratic model
    \(f(x)=\frac12 x^\top\diag(1,\kappa)x\).  The induced condition number is that of
    \(M^{-1/2}HM^{-1/2}\).}
    \label{tab:intro-two-d}
    \small
    \begin{tabular}{@{}p{.30\textwidth}p{.25\textwidth}p{.32\textwidth}@{}}
    \toprule
    Metric \(M\) & Induced condition number & Semi-axis lengths \\
    \midrule
    \(I\) & \(\kappa\) & \(\Delta,\Delta\) \\
    \(\diag(\kappa^{-1/2},\kappa^{1/2})\)
        & \(1\) & \(\Delta\kappa^{1/4},\Delta\kappa^{-1/4}\) \\
    \(\diag(K^{-1/2},K^{1/2})\), \(K\le\kappa\)
        & \(\kappa/K\) & \(\Delta K^{1/4},\Delta K^{-1/4}\) \\
    \bottomrule
    \end{tabular}
\end{table}

The paper develops \MATRO{} (Metric-Aware Trust-Region Optimization) around these
ideas.  The first part establishes the metric trust-region framework: metric balls,
induced-coordinate subproblems, poisedness in metric balls, and a metric contract that
keeps the analysis uniform.  The second part constructs metrics from interpolation
curvature.  For positive definite quadratics, the determinant-normalized Hessian is the
unique volume-normalized metric that makes the induced Hessian isotropic; moreover,
the corresponding metric trust-region step is a truncated Newton step.  For indefinite
Hessians, an absolute-curvature metric balances curvature magnitudes while preserving
the signs of the model curvature.  Stabilized spectral variants satisfy the metric contract
needed for the standard fully quadratic convergence proof.  The numerical experiments
then examine the expected cost--accuracy trade-off: curvature-shaped regions can help
when the fitted Hessian captures persistent anisotropy, but they require full quadratic
models and dense linear algebra.

\subsection{Related work}
\label{subsec:related}
The algorithmic setting is the model-based DFO trust-region framework.  The monograph
of Conn, Scheinberg, and Vicente \cite{ConnScheinbergVicente2009} develops fully
linear and fully quadratic model theory, criticality steps, geometry-improvement
mechanisms, and evaluation-complexity analysis for interpolation-based DFO.  The
classical derivative-based trust-region background is given by Conn, Gould, and Toint
\cite{ConnGouldToint2000}, and the geometry of interpolation and regression sets is
analyzed in \cite{ConnScheinbergVicente2008a,ConnScheinbergVicente2008b}.  Powell's
solvers are central practical references: UOBYQA uses full quadratic interpolation
\cite{Powell2002UOBYQA}, whereas NEWUOA and BOBYQA reduce the number of points
through minimum-change Hessian updates \cite{Powell2006NEWUOA,Powell2009BOBYQA}.
PRIMA provides a modern reference implementation of Powell's methods
\cite{Zhang2023PRIMA}.

A recent line of work has emphasized that the model-construction rule itself is a
mathematical design choice.  Examples include underdetermined models derived from
trust-region iteration properties \cite{XieYuan2025SIAM,LiZhouXieLi2025,YeLiXieYu2025,XieYuan2023}, least \(H^2\)-norm updating
\cite{XieYuan2026H2}, regional minimal updating \cite{XieWild2026ReMU},
transformed-objective DFO \cite{XieYuan2025DFOTO}, and two-dimensional model-based
subspace methods \cite{XieYuan2026MoSub}.  These works modify the norm, update
criterion, transformation, or subspace used to construct a local model.  \MATRO{} is
complementary to this direction.  It keeps the model fully quadratic and treats a
different object as design variable: the metric defining where the model is trusted.

Scaling and preconditioning are classical in smooth trust-region methods
\cite{ConnGouldToint2000}.  Their derivative-free use is subtler because the relevant
local geometry is not normally available from exact derivatives.  A fixed change of
variables can remove known parameter scales, but it does not generally follow a local
curvature geometry that rotates or changes during the run.  Other DFO algorithms adapt
geometry in different ways: MADS adapts polling directions in a direct-search framework
\cite{AudetDennis2006MADS}, RBF trust-region methods build nonpolynomial local
surrogates \cite{WildShoemaker2011}, and CMA-ES adapts a covariance matrix in a
population-based stochastic search \cite{HansenMullerKoumoutsakos2003}.  The present
work remains within deterministic interpolation-based trust regions and asks how the
curvature already present in a fully quadratic model can be converted into a stable
ellipsoidal trust-region shape.  Three features distinguish \MATRO{} from
Hessian-scaled trust regions in the smooth setting: (i)~the metric is derived from
an \emph{interpolation} Hessian that may be inaccurate or indefinite, not from an
exact gradient or Hessian; (ii)~the metric defines only the trust-region geometry,
while the signed Hessian is retained in the subproblem objective---the model is not
convexified; and (iii)~the metric changes at every iteration and must satisfy a
uniform spectral contract to preserve the fully quadratic convergence machinery.

\subsection{Notation}
\label{subsec:notation}
Throughout the paper, \(\R^n\) is equipped with the standard inner product
\(\ip{u}{v}=u^\top v\) and Euclidean norm \(\norm{u}\).  For \(M\succ0\), define
\(\normM{u}{M}=(u^\top Mu)^{1/2}\), and let \(M^{1/2}\) denote the SPD square root.
For a symmetric matrix \(H\), define
\begin{equation}
    \normM{H}{M}:=\norm{M^{-1/2}HM^{-1/2}}_2.
    \label{eq:HnormM}
\end{equation}
If \(H\succ0\), write
\begin{equation}
    \kappa_M(H):=\cond(M^{-1/2}HM^{-1/2}).
    \label{eq:kappaM-def}
\end{equation}
The condition number is \(\cond(M)=\lammax(M)/\lammin(M)\), and \(A\preceq B\)
means that \(B-A\) is positive semidefinite.  At iteration \(k\), \(x_k\) is the current
point, \(\Delta_k\) is the trust-region radius, \(M_k\) is the metric, and
\(T_k=M_k^{1/2}\).  In induced coordinates \(y=T_ks\), the model gradient and Hessian
are denoted by \(g_{y,k}\) and \(B_{y,k}\).  We use
\begin{equation}
    \chi_k=\norm{g_{y,k}},
    \qquad
    \chi_k^{\rm true}=\norm{M_k^{-1/2}\nabla f(x_k)}
    \label{eq:chi-def-intro}
\end{equation}
for the model and true metric-dual stationarity measures.

\section{Metric Trust-Region Framework}
\label{sec:framework}

The metric will later be selected from interpolation curvature, but the basic
trust-region mechanism does not depend on that choice.  This section fixes an arbitrary
positive definite metric and rewrites the usual trust-region ingredients in the induced
coordinates where the metric ball is Euclidean.  The identities below are elementary,
but they are the mechanism that allows the later analysis to separate trust-region
globalization from metric selection.

\subsection{Metric balls and volume}
\label{subsec:metric-balls}
Given \(x\in\R^n\), \(\Delta>0\), and \(M\succ0\), define the metric ball
\begin{equation}
    \Bball_M(x,\Delta)
    :=\{x+s:\normM{s}{M}\le \Delta\}.
    \label{eq:M-ball-def}
\end{equation}
If \(M=Q\diag(\mu_1,\ldots,\mu_n)Q^\top\), with \(Q\) orthogonal and \(\mu_i>0\),
then \(\Bball_M(x,\Delta)\) is an ellipsoid with principal directions given by the
columns of \(Q\) and semi-axis lengths \(\Delta/\sqrt{\mu_i}\).  Its volume is
\begin{equation}
    \vol\bigl(\Bball_M(x,\Delta)\bigr)
    = \omunit\Delta^n\det(M)^{-1/2},
    \label{eq:M-ball-volume}
\end{equation}
where \(\omunit\) is the volume of the Euclidean unit ball in \(\R^n\).

\begin{proposition}[Axes and volume of a metric ball]
\label{prop:axes-volume}
Let \(M\succ0\) have the spectral decomposition
\(M=Q\diag(\mu_1,\ldots,\mu_n)Q^\top\).
Then \(\Bball_M(x,\Delta)\) has semi-axis lengths
\(\Delta/\sqrt{\mu_i}\) along~\(Qe_i\), and its volume is
given by \eqref{eq:M-ball-volume}.  If \(\det(M)=1\),
then \(\Bball_M(x,\Delta)\) and \(B(x,\Delta)\) have the
same volume; the matrix~\(M\) changes the shape but not
the volume scale.
\end{proposition}

\begin{proof}
Write \(s=Qz\).  Then \(s^\top Ms=\sum_i \mu_i z_i^2\).  The boundary condition
\(s^\top Ms=\Delta^2\) gives the semi-axis lengths \(\Delta/\sqrt{\mu_i}\).  The change
of variables \(y=M^{1/2}s\) maps the ellipsoid onto the Euclidean ball \(\norm{y}\le
\Delta\) and has Jacobian determinant \(\det(M)^{1/2}\).  Hence
\(\vol(\Bball_M)=\det(M)^{-1/2}\omunit\Delta^n\).
\end{proof}

\Cref{prop:axes-volume} explains the determinant normalization used in Section~\ref{sec:metric}.
The trust-region radius already controls the overall size of the trusted region.  The
metric is used to distribute this size among directions.  Setting \(\det(M)=1\) is a
convenient way to keep this separation explicit.

\subsection{Subproblems in induced variables}
\label{subsec:affine-equivalence}
Let \(M\succ0\), \(T=M^{1/2}\), and \(y=Ts\).  The model in the original step variable
is
\begin{equation}
    m_x(s)=f(x)+g_x^\top s + \frac12 s^\top H_x s.
    \label{eq:model-x}
\end{equation}
Define the induced model \(\widetilde m\) by
\begin{equation}
    \widetilde m(y)=m_x(T^{-1}y).
    \label{eq:model-y-def}
\end{equation}

\begin{theorem}[Equivalence of metric and Euclidean subproblems]
\label{thm:affine-equiv}
Let \(M\succ0\), \(T=M^{1/2}\), and \(y=Ts\).  The two optimization problems
\begin{equation}
    \min_{\normM{s}{M}\le \Delta} m_x(s),
    \qquad
    \min_{\norm{y}\le \Delta} \widetilde m(y)
    \label{eq:two-subproblems}
\end{equation}
are equivalent under the change of variables \(y=Ts\) and have the same optimal value.  Moreover, predicted reduction is preserved:
\begin{equation}
    m_x(0)-m_x(s)=\widetilde m(0)-\widetilde m(y),
    \qquad y=Ts,
    \label{eq:pred-preserved}
\end{equation}
and the gradient and Hessian in induced variables are
\begin{equation}
    g_y=T^{-1}g_x,
    \qquad
    B_y=T^{-1}H_xT^{-1}.
    \label{eq:gy-By}
\end{equation}
\end{theorem}

\begin{proof}
Since \(T\) is SPD, \(\normM{s}{M}^2=s^\top T^2s=\norm{Ts}^2=\norm{y}^2\).  Hence
feasibility is preserved by \(y=Ts\).  Substitution gives
\begin{equation}
    m_x(T^{-1}y)=f(x)+g_x^\top T^{-1}y+\frac12 y^\top T^{-1}H_xT^{-1}y.
\end{equation}
The equality of optimal values and predicted reductions follows immediately, and the
expressions in \eqref{eq:gy-By} follow by differentiation.
\end{proof}

The standard Cauchy decrease condition can also be written directly in the metric of
the original variables.  This form makes clear how the choice of \(M\) enters the
first-order decrease mechanism.
For a symmetric matrix \(H\), write
\[
    \norm{H}_{M}:=\norm{M^{-1/2}HM^{-1/2}}_2 .
\]

\begin{proposition}[Metric form of Cauchy decrease]
\label{prop:metric-cauchy}
Let \(m_x\) be given by \eqref{eq:model-x}.  Suppose that the induced step \(y\) satisfies
the usual Cauchy decrease condition for \(\widetilde m\) on \(\norm{y}\le\Delta\):
\begin{equation}
    \widetilde m(0)-\widetilde m(y)
    \ge \frac12 \norm{g_y}
    \min\left\{\Delta,\frac{\norm{g_y}}{\norm{B_y}_2}\right\},
    \label{eq:y-cauchy}
\end{equation}
with the convention that the quotient is \(+\infty\) when \(B_y=0\).  Then
\(s=T^{-1}y\) satisfies
\begin{equation}
    m_x(0)-m_x(s)
    \ge \frac12 \norm{g_x}_{M^{-1}}
    \min\left\{\Delta,
    \frac{\norm{g_x}_{M^{-1}}}{\norm{H_x}_{M}}
    \right\}.
    \label{eq:metric-cauchy}
\end{equation}
\end{proposition}

\begin{proof}
By \eqref{eq:gy-By},
\[
  \norm{g_y}=\norm{M^{-1/2}g_x}=\norm{g_x}_{M^{-1}},
  \qquad
  \norm{B_y}_2=\norm{M^{-1/2}H_xM^{-1/2}}_2=\norm{H_x}_{M}.
\]
The predicted reduction identity \eqref{eq:pred-preserved}
then turns \eqref{eq:y-cauchy} into \eqref{eq:metric-cauchy}.
\end{proof}

\subsection{Poisedness in metric balls}
\label{subsec:M-poised}
At iteration \(k\), the modeled function in induced variables is
\begin{equation}
    \widetilde f_k(y)=f(x_k+T_k^{-1}y),
    \qquad \norm{y}\le \Delta_k,
    \label{eq:induced-f}
\end{equation}
where \(T_k=M_k^{1/2}\).  Let
\begin{equation}
    q_n=\frac{(n+1)(n+2)}{2}
    \label{eq:qn-def}
\end{equation}
be the dimension of the quadratic polynomial space.  A full quadratic interpolation
model in induced variables has the form
\begin{equation}
    \widetilde m_k(y)=f(x_k)+g_{y,k}^\top y+\frac12 y^\top B_{y,k}y.
    \label{eq:model-induced}
\end{equation}
The trial step is computed from
\begin{equation}
    \min_{\norm{y}\le\Delta_k}
    g_{y,k}^\top y+\frac12 y^\top B_{y,k}y,
    \label{eq:y-subproblem}
\end{equation}
and the original step is \(s_k=T_k^{-1}y_k\).  If the model is expressed in the original
variables, its Hessian is
\begin{equation}
    H_{x,k}=T_k B_{y,k}T_k.
    \label{eq:Hx-from-By}
\end{equation}
This matrix is the curvature information passed to the metric-selection routine.

We use the standard notion of \(\Lambda\)-poisedness in Euclidean coordinates, but apply
it after the induced-coordinate map.  This convention matters because changing
\(M_k\) changes the coordinates in which archived points are viewed.

\begin{definition}[Poisedness in a metric ball]
\label{def:M-poised}
Let \(M\succ0\), \(T=M^{1/2}\), and let
\(Y=\{x+s^{(0)},\ldots,x+s^{(q_n-1)}\}\subset\Bball_M(x,\Delta)\).  We say that
\(Y\) is \(\Lambda\)-poised in \(\Bball_M(x,\Delta)\) if the transformed set
\begin{equation}
    \widetilde Y=\bigl\{T s^{(0)},\ldots,T s^{(q_n-1)}\bigr\}
    \label{eq:transformed-set}
\end{equation}
is \(\Lambda\)-poised in the Euclidean ball \(B(0,\Delta)\) for quadratic interpolation.
\end{definition}

Metric poisedness is the usual poisedness condition applied in the variables in which
the trust region is a ball.  The classical fully quadratic interpolation theory can
therefore be used once all quantities are written in induced coordinates.

\subsection{The abstract algorithm}
\label{subsec:abstract-alg}
The metric update is written abstractly as
\begin{equation}
    M_{k+1}=\MetSel(H_{x,k},M_k,A_k,\Delta_k;\theta),
    \label{eq:metric-selection-abstract}
\end{equation}
where \(A_k\) is the archive of sampled points and \(\theta\) denotes fixed metric
parameters.  The convergence proof uses only the following contract.

\begin{assumption}[Uniform metric bounds]
\label{asmp:metric-contract}
There exist constants \(0<m_{\min}\le m_{\max}<\infty\) such that every metric returned
by \(\MetSel\) satisfies
\begin{equation}
    m_{\min}I\preceq M_k\preceq m_{\max}I,
    \qquad k=0,1,2,\ldots .
    \label{eq:metric-contract}
\end{equation}
\end{assumption}

\Cref{alg:abstract-matro} is the abstract method analyzed in Section~\ref{sec:convergence}.
The specific spectral rule used in the experiments is introduced in Section~\ref{sec:metric}.
Implementation safeguards used in the numerical section, such as a lower radius bound
or adaptive initialization, are not part of the complexity proof.

\begin{algorithm}[t]
\caption{Abstract \MATRO{} method}
\label{alg:abstract-matro}
\begin{algorithmic}[1]
\Require Objective \(f\), initial point \(x_0\), radius \(\Delta_0>0\), parameters
\(0<\eta_1\le\eta_2<1\), \(0<\gamma_{\rm dec}<1<\gamma_{\rm inc}\), criticality constant \(\mu_{\rm crit}\), and a \MetSel{} routine.  The criticality threshold is normalized to one.
\State Set \(M_0=I\), \(T_0=I\), and initialize the archive with \((x_0,f(x_0))\).
\State Build an initial \(\Lambda\)-poised set of \(q_n\) points in induced coordinates;
evaluate \(f\) and add the points to the archive.
\For{\(k=0,1,2,\ldots\)}
    \Repeat
        \State Select or repair a \(\Lambda\)-poised interpolation set in \(\norm{y}\le\Delta_k\);
        fit \(\widetilde m_k\) in \eqref{eq:model-induced}.
        \If{\(\norm{g_{y,k}}\le1\) and \(\Delta_k>\mu_{\rm crit}\norm{g_{y,k}}\)}
            \State Set \(\Delta_k\leftarrow \mu_{\rm crit}\norm{g_{y,k}}\).
        \EndIf
    \Until{\(\norm{g_{y,k}}>1\) or \(\Delta_k\le\mu_{\rm crit}\norm{g_{y,k}}\)}
    \State Approximately solve \eqref{eq:y-subproblem} to obtain \(y_k\) satisfying the Cauchy decrease condition.
    \State Set \(s_k=T_k^{-1}y_k\); evaluate \(f(x_k+s_k)\) and add the trial point to the archive.
    \State Compute \(\pred_k=\widetilde m_k(0)-\widetilde m_k(y_k)\) and
    \(\ared_k=f(x_k)-f(x_k+s_k)\).
    \State Set \(\rho_k=\ared_k/\pred_k\) if \(\pred_k>0\), and \(\rho_k=-\infty\) otherwise.
    \State Accept \(x_{k+1}=x_k+s_k\) if \(\rho_k\ge\eta_1\); otherwise set \(x_{k+1}=x_k\).
    \State Update \(\Delta_{k+1}\) by the standard trust-region rule.
    \State Form \(H_{x,k}=T_kB_{y,k}T_k\), set
    \(M_{k+1}=\MetSel(H_{x,k},M_k,A_k,\Delta_k;\theta)\), and set
    \(T_{k+1}=M_{k+1}^{1/2}\).
\EndFor
\end{algorithmic}
\end{algorithm}

\section{Metrics from Interpolation Curvature}
\label{sec:metric}

We now turn from the abstract metric framework to the construction of the metric.
The guiding object is the Hessian of the fully quadratic interpolation model.  When the
interpolation geometry is adequate and the radius is small, this Hessian approximates the
true local Hessian; before that asymptotic regime is reached, it remains the only
second-order object available to a fully quadratic DFO method.  The analysis in this
section explains how much one can gain from this curvature, what is lost when the
metric must be capped for stability, and how the same idea should be interpreted when
the fitted Hessian is indefinite.

\subsection{Induced conditioning as a metric criterion}
\label{subsec:metric-objective}
For an SPD curvature matrix \(H\), a natural measure of the anisotropy left after choosing
\(M\succ0\) is
\begin{equation}
    \kappa_M(H)=\cond(M^{-1/2}HM^{-1/2}).
    \label{eq:metric-objective-cond}
\end{equation}
This quantity is the condition number of the quadratic model in induced variables.
Since multiplying \(M\) by a positive scalar can be absorbed into the radius, we fix the
scale by imposing \(\det(M)=1\).  The idealized design problem for a positive definite
quadratic model is therefore
\begin{equation}
    \hbox{choose } M\succ0,\quad \det(M)=1,
    \quad \hbox{so that } \cond(M^{-1/2}HM^{-1/2}) \hbox{ is small.}
    \label{eq:metric-design-problem}
\end{equation}
The next results follow this design problem from the ideal SPD case to the stabilized
rule used by the algorithm.  The order is deliberate: first the exact quadratic calculation
identifies the metric one would choose if the curvature were known and positive
definite; then the cap, approximation error, and indefinite-curvature cases explain the
modifications needed in a derivative-free implementation.

\subsection{Positive curvature and Newton rays}
\label{subsec:exact-SPD}
The exact SPD quadratic case provides the reference point for the whole construction.
Here the curvature matrix is both a reliable local model Hessian and a valid positive
definite shape matrix.  The determinant constraint removes irrelevant scaling and leaves
a pure shape-selection problem.

\begin{proposition}[Optimal metric for positive definite quadratics]
\label{prop:optimal-det-metric}
Let \(H\succ0\) and define
\begin{equation}
    \Mclass_{\det}:=\{M\succ0:\det(M)=1\}.
\end{equation}
Then
\begin{equation}
    M_\star=\frac{H}{\det(H)^{1/n}}
    \label{eq:Mstar-general}
\end{equation}
belongs to \(\Mclass_{\det}\) and satisfies
\begin{equation}
    M_\star^{-1/2}HM_\star^{-1/2}=\det(H)^{1/n} I.
    \label{eq:Mstar-isotropic-general}
\end{equation}
Consequently, \(\kappa_{M_\star}(H)=1\), which is the minimum possible induced
condition number.  The minimizer is unique under the determinant constraint.
\end{proposition}

\begin{proof}
Let \(\alpha=\det(H)^{1/n}\).  Then \(M_\star=H/\alpha\) and
\(\det(M_\star)=\det(H)/\alpha^n=1\).  Moreover
\(M_\star^{-1/2}=\alpha^{1/2}H^{-1/2}\), so
\(M_\star^{-1/2}HM_\star^{-1/2}=\alpha I\).  No SPD matrix can produce an induced
condition number below one.  If another \(M\in\Mclass_{\det}\) also gives condition
number one, then \(M^{-1/2}HM^{-1/2}=\beta I\) for some \(\beta>0\), hence
\(H=\beta M\).  Taking determinants gives \(\det(H)=\beta^n\det(M)=\beta^n\), so
\(\beta=\det(H)^{1/n}\) and \(M=H/\beta=M_\star\).
\end{proof}

The preceding proposition describes the shape of the quadratic in induced variables.
The next proposition translates that shape statement into the actual step.  With the
curvature metric, the constrained step follows the Newton ray and is shortened only when
the metric radius cuts off the full Newton step.

\begin{proposition}[Newton-ray property of the curvature metric]
\label{prop:newton-ray}
Let
\begin{equation}
    m(s)=f_0+g^\top s+\frac12s^\top Hs,
    \qquad H\succ0,
    \label{eq:spd-model-newton-ray}
\end{equation}
and let \(M=H/\alpha\), where \(\alpha=\det(H)^{1/n}\).  The solution of
\begin{equation}
    \min_{s^\top Ms\le \Delta^2} g^\top s+\frac12s^\top Hs
    \label{eq:metric-spd-subproblem}
\end{equation}
is
\begin{equation}
    s_\Delta=-t_\Delta H^{-1}g,
    \qquad
    t_\Delta=
    \min\left\{1,\ \Delta\sqrt{\frac{\alpha}{g^\top H^{-1}g}}\right\},
    \label{eq:truncated-newton-step}
\end{equation}
with the convention that \(s_\Delta=0\) if \(g=0\).  Moreover,
\begin{equation}
    m(0)-m(s_\Delta)
    =t_\Delta\left(1-\frac{t_\Delta}{2}\right)g^\top H^{-1}g.
    \label{eq:newton-ray-pred}
\end{equation}
\end{proposition}

\begin{proof}
If \(g=0\), the claim is immediate.  Assume \(g\ne0\).  In the induced variables
\(y=M^{1/2}s\), the Hessian is \(M^{-1/2}HM^{-1/2}=\alpha I\) by
\cref{prop:optimal-det-metric}, and the gradient is
\(g_y=M^{-1/2}g=\alpha^{1/2}H^{-1/2}g\).  Hence the induced subproblem is
\begin{equation}
    \min_{\norm{y}\le \Delta} g_y^\top y+\frac{\alpha}{2}\norm{y}^2.
\end{equation}
Its unconstrained minimizer is \(-\alpha^{-1}g_y\), whose norm is
\(\sqrt{g^\top H^{-1}g/\alpha}\).  If this point is feasible, then \(t_\Delta=1\).  Otherwise
the minimizer is the boundary point in the direction \(-g_y\), which corresponds to
\(t_\Delta=\Delta\sqrt{\alpha/(g^\top H^{-1}g)}\).  Mapping back by
\(s=M^{-1/2}y\) gives \eqref{eq:truncated-newton-step}.  Substitution into
\eqref{eq:spd-model-newton-ray} gives \eqref{eq:newton-ray-pred}.
\end{proof}

\begin{corollary}[One-step recovery for exact positive definite quadratics]
\label{cor:exact-quadratic-recovery}
Let \(f(x)=f_\star+\frac12(x-x_\star)^\top H(x-x_\star)\) with \(H\succ0\).  Suppose
that the interpolation model is exact at \(x_k\), and let
\(g_k=\nabla f(x_k)\) and \(M_k=H/\det(H)^{1/n}\).  If
\begin{equation}
    \Delta_k\ge
    \sqrt{\frac{g_k^\top H^{-1}g_k}{\det(H)^{1/n}}},
    \label{eq:newton-feasible-condition}
\end{equation}
then the metric trust-region step gives \(x_{k+1}=x_\star\).  If
\eqref{eq:newton-feasible-condition} fails, the step lies on the Newton ray
\(x_k-tH^{-1}g_k\) with \(0<t<1\).
\end{corollary}

\begin{proof}
Apply \cref{prop:newton-ray} with \(g=g_k\).  The unconstrained Newton step
\(-H^{-1}g_k\) reaches \(x_\star\).  The feasibility condition is exactly
\eqref{eq:newton-feasible-condition}.
\end{proof}

The same calculation also gives the Cauchy decrease constant used below.
For \(M=H/\alpha\), \(\norm{H}_M=\alpha\) and
\(\norm{g}_{M^{-1}}^2=\alpha g^\top H^{-1}g\).  Thus the metric Cauchy decrease in
\cref{prop:metric-cauchy} is expressed through the Newton decrement
\(g^\top H^{-1}g\) and the geometric-mean curvature \(\alpha\), instead of through the
largest eigenvalue of \(H\).  This is one way in which the metric changes constants
without changing the standard trust-region order of complexity.

\subsection{The effect of a metric cap}
\label{subsec:capping}
In practice, a metric with a very large condition number is undesirable.  It can make
archived interpolation points poorly distributed in induced coordinates and can amplify
roundoff errors.  A condition-number cap restricts the amount of anisotropy allowed in
the trust region.  The next result shows that such a cap also imposes a fundamental
limit: no metric with \(\cond(M)\le K\) can remove more than a factor \(K\) of the
condition number of an SPD quadratic.

\begin{proposition}[Conditioning limit under a metric cap]
\label{prop:cap-lower-bound}
Let \(H\succ0\) and let \(M\succ0\) satisfy \(\cond(M)\le K\).  Then
\begin{equation}
    \cond(M^{-1/2}HM^{-1/2})\ge \frac{\cond(H)}{K}.
    \label{eq:cap-lower-bound}
\end{equation}
\end{proposition}

\begin{proof}
Using the generalized Rayleigh quotient,
\begin{equation}
    \lammax(M^{-1/2}HM^{-1/2})
    =\max_{v\ne0}\frac{v^\top Hv}{v^\top Mv}
    \ge \frac{\lammax(H)}{\lammax(M)},
\end{equation}
and
\begin{equation}
    \lammin(M^{-1/2}HM^{-1/2})
    =\min_{v\ne0}\frac{v^\top Hv}{v^\top Mv}
    \le \frac{\lammin(H)}{\lammin(M)}.
\end{equation}
Taking the ratio gives
\begin{equation}
    \cond(M^{-1/2}HM^{-1/2})
    \ge \frac{\lammax(H)}{\lammin(H)}\frac{\lammin(M)}{\lammax(M)}
    =\frac{\cond(H)}{\cond(M)}
    \ge \frac{\cond(H)}{K}.
\end{equation}
\end{proof}

For the two-dimensional quadratic in \cref{tab:intro-two-d}, the bound is attained by
the capped metric \(M_K=\diag(K^{-1/2},K^{1/2})\).  Thus the cap is not only a
numerical safeguard; it also quantifies the maximum conditioning improvement that the
chosen metric class can express.

\subsection{Metric accuracy from interpolation curvature}
\label{subsec:alignment}
The interpolation Hessian is only an approximation.  We first record a perturbation statement: if a metric is close to a curvature shape in Loewner order, then the induced Hessian is well conditioned.

\begin{proposition}[Induced conditioning from Loewner alignment]
\label{prop:alignment}
Let \(M\succ0\) and \(H_{\rm shape}\succ0\).  Suppose there exist \(\alpha>0\) and
\(\delta\ge0\) such that
\begin{equation}
    e^{-\delta}\alpha M \preceq H_{\rm shape}\preceq e^{\delta}\alpha M.
    \label{eq:alignment-assumption}
\end{equation}
Then
\begin{equation}
    \cond(M^{-1/2}H_{\rm shape}M^{-1/2})\le e^{2\delta}.
    \label{eq:alignment-result}
\end{equation}
\end{proposition}

\begin{proof}
Premultiplying and postmultiplying \eqref{eq:alignment-assumption} by \(M^{-1/2}\)
gives
\begin{equation}
    e^{-\delta}\alpha I
    \preceq M^{-1/2}H_{\rm shape}M^{-1/2}
    \preceq e^{\delta}\alpha I.
\end{equation}
All eigenvalues of the induced matrix lie in \([e^{-\delta}\alpha,e^\delta\alpha]\), and
their ratio is at most \(e^{2\delta}\).
\end{proof}

For a fully quadratic model, this alignment estimate can be read directly from the
Hessian error bound.

\begin{corollary}[Fully quadratic Hessian error and metric accuracy]
\label{cor:fq-hessian-metric}
Let \(H_\star\succ0\) with \(\lammin(H_\star)\ge \nu>0\), and let \(H\succ0\) satisfy
\begin{equation}
    \norm{H-H_\star}_2\le \kappa_H\Delta,
    \qquad
    \theta:=\frac{\kappa_H\Delta}{\nu}<1.
    \label{eq:fq-hessian-error-metric}
\end{equation}
Let \(M=H/\det(H)^{1/n}\).  Then
\begin{equation}
    \cond(M^{-1/2}H_\star M^{-1/2})
    \le \frac{1+\theta}{1-\theta}.
    \label{eq:fq-induced-cond}
\end{equation}
\end{corollary}

\begin{proof}
Since \(\lammin(H_\star)\ge\nu\),
\begin{equation}
    \norm{H_\star^{-1/2}(H-H_\star)H_\star^{-1/2}}_2
    \le \frac{\norm{H-H_\star}_2}{\nu}
    \le \theta.
\end{equation}
Thus
\begin{equation}
    (1-\theta)H_\star\preceq H\preceq (1+\theta)H_\star,
\end{equation}
and hence
\begin{equation}
    \frac{1}{1+\theta}H\preceq H_\star\preceq \frac{1}{1-\theta}H.
\end{equation}
Writing \(H=\alpha M\), with \(\alpha=\det(H)^{1/n}\), gives
\[
    \frac{\alpha}{1+\theta}M
    \preceq
    H_\star
    \preceq
    \frac{\alpha}{1-\theta}M .
\]
Premultiplying and postmultiplying by \(M^{-1/2}\), all eigenvalues of
\(M^{-1/2}H_\star M^{-1/2}\) lie in
\[
    \left[\frac{\alpha}{1+\theta},\frac{\alpha}{1-\theta}\right].
\]
Therefore
\[
    \cond(M^{-1/2}H_\star M^{-1/2})
    \le
    \frac{1+\theta}{1-\theta}.
\]
\end{proof}

\Cref{cor:fq-hessian-metric} is not used as an additional assumption in the convergence
proof.  It explains the design: when the fully quadratic Hessian is relatively accurate
for a locally positive definite true Hessian, the metric obtained from it nearly
isotropizes the true local curvature.

\subsection{Indefinite curvature and signed scaling}
\label{subsec:indefinite}
An interpolation Hessian need not be positive definite.  The metric, however, must be
SPD.  Negative curvature should remain in the quadratic model used by the
trust-region subproblem; it should not become a negative length scale in the region.
The ideal nonsingular indefinite case can also be described exactly.

\begin{proposition}[Signed isotropization by the absolute-curvature metric]
\label{prop:signed-isotropization}
Let \(H=V\diag(\lambda_1,\ldots,\lambda_n)V^\top\) be nonsingular and symmetric, and
let
\begin{equation}
    |H|:=V\diag(|\lambda_1|,\ldots,|\lambda_n|)V^\top,
    \qquad
    M_{\rm abs}:=\frac{|H|}{\det(|H|)^{1/n}}.
    \label{eq:Mabs-ideal}
\end{equation}
Then
\begin{equation}
    M_{\rm abs}^{-1/2}HM_{\rm abs}^{-1/2}
    = \det(|H|)^{1/n}
    V\diag(\operatorname{sign}(\lambda_1),\ldots,\operatorname{sign}(\lambda_n))V^\top.
    \label{eq:signed-isotropization}
\end{equation}
\end{proposition}

\begin{proof}
Let \(\alpha=\det(|H|)^{1/n}\).  Since
\(M_{\rm abs}^{-1/2}=\alpha^{1/2}V\diag(|\lambda_i|^{-1/2})V^\top\), substitution gives
\eqref{eq:signed-isotropization}.
\end{proof}

Thus the metric equalizes curvature magnitudes in induced variables, while the signs of
the curvature directions are preserved in the model Hessian.  An absolute-curvature shape changes only the length scale; it does not convexify the trust-region model.  In practice,
the fitted Hessian may be singular or nearly singular, so the construction uses a spectral
floor.

Concretely, for an eigendecomposition
\begin{equation}
    H_{x,k}=V\diag(\lambda_1,\ldots,\lambda_n)V^\top,
    \label{eq:Hx-eig}
\end{equation}
we form magnitudes
\begin{equation}
    a_i=\max\{\abs{\lambda_i},\sigma\},
    \qquad i=1,\ldots,n,
    \label{eq:curv-magnitudes}
\end{equation}
where \(\sigma>0\) is a spectral floor.  The resulting proxy
\begin{equation}
    H_{\rm abs,k}:=V\diag(a_1,\ldots,a_n)V^\top
    \label{eq:Habs-def}
\end{equation}
is SPD.  It should be interpreted as a curvature-shape proxy extracted from the fitted
Hessian, not as a replacement for the signed Hessian in the subproblem.

\subsection{A bounded spectral construction}
\label{subsec:spectral-rule}
The spectral rule used in the experiments converts \(H_{\rm abs,k}\) into a bounded
volume-normalized metric.  Starting from \eqref{eq:curv-magnitudes}, cap the ratio of
largest to smallest entries by setting
\begin{equation}
    \bar a_i=\max\left\{a_i,\ \kappa_{\max}^{-1}\max_j a_j\right\},
    \qquad i=1,\ldots,n .
    \label{eq:cap-ai}
\end{equation}
Then normalize the determinant:
\begin{equation}
    \widehat a_i=\frac{\bar a_i}{\left(\prod_{j=1}^n \bar a_j\right)^{1/n}},
    \qquad
    S_k=V\diag(\widehat a_1,
    \ldots,
    \widehat a_n)V^\top.
    \label{eq:S-k-def}
\end{equation}
Thus \(S_k\succ0\), \(\det(S_k)=1\), and \(\cond(S_k)\le\kappa_{\max}\).  To avoid
abrupt changes in induced coordinates, the final metric is chosen from
\begin{equation}
\begin{split}
    \Cclass_k:=\{M\succ0:
    &\ \det(M)=1,
    \ \cond(M)\le \kappa_{\max}, \\
    &\ e^{-\delta_M}M_k\preceq M\preceq e^{\delta_M}M_k\}.
\end{split}
    \label{eq:Ck-class}
\end{equation}
The final metric is obtained from \(S_k\) by taking a damped
affine-invariant geodesic step from \(M_k\) toward \(S_k\).

\begin{enumerate}[label=\textbf{(S\arabic*)}]
\item \emph{Relative eigendecomposition.}
Compute the symmetric relative matrix
\[
    C_k=M_k^{-1/2}S_kM_k^{-1/2}
       =P\diag(\mu_1,\ldots,\mu_n)P^\top,
    \qquad P^\top P=I .
\]
\item \emph{Geodesic damping.}
Set
\[
    \ell_i=\log \mu_i,\qquad
    L_k=\max_{1\le i\le n}|\ell_i|,
\]
and define
\[
    \omega_k=
    \begin{cases}
    1, & L_k=0,\\[2mm]
    \min\{1,\delta_M/L_k\}, & L_k>0.
    \end{cases}
\]
\item \emph{Metric update.}
Return
\[
    M_{k+1}
    =
    M_k^{1/2}
    P\diag\!\left(e^{\omega_k\ell_1},\ldots,e^{\omega_k\ell_n}\right)
    P^\top
    M_k^{1/2}.
\]
\end{enumerate}

\begin{proposition}[Output of the spectral construction]\label{prop:spectral-output}
Suppose \(M_k\succ0\) with \(\det(M_k)=1\) and \(\cond(M_k)\le\kappa_{\max}\),
and let \(S_k\) be defined by \eqref{eq:cap-ai}--\eqref{eq:S-k-def}.  Then the
metric \(M_{k+1}\) produced by steps \textup{(S1)--(S3)} satisfies
\(M_{k+1}\in\Cclass_k\).
\end{proposition}

\begin{proof}
By \eqref{eq:S-k-def}, \(S_k\succ0\), \(\det(S_k)=1\), and
\(\cond(S_k)\le\kappa_{\max}\).  Since \(\det(M_k)=1\), the relative matrix
\(C_k=M_k^{-1/2}S_kM_k^{-1/2}\) satisfies \(\det(C_k)=1\).  Hence
\[
    \sum_{i=1}^n \ell_i=\log\det(C_k)=0 .
\]
The update in \textup{(S3)} has relative eigenvalues
\(e^{\omega_k\ell_i}\) with respect to \(M_k\).  Therefore
\[
    \det(M_{k+1})
    =
    \det(M_k)\prod_{i=1}^n e^{\omega_k\ell_i}
    =
    1
\]
and, by the definition of \(\omega_k\),
\[
    |\omega_k\ell_i|\le \delta_M
    \qquad\text{for all }i.
\]
Thus
\[
    e^{-\delta_M}M_k\preceq M_{k+1}\preceq e^{\delta_M}M_k .
\]

We now verify the condition-number cap.  For this purpose, note that the
update is the weighted affine-invariant geometric mean
\[
    M_{k+1}=M_k\#_{\omega_k}S_k .
\]
For any \(A,B\succ0\) and \(t\in[0,1]\), monotonicity and homogeneity of the
weighted geometric mean give
\[
    \lambda_{\max}(A\#_t B)
    \le
    \lambda_{\max}(A)^{1-t}\lambda_{\max}(B)^t,
\]
and, using \((A\#_t B)^{-1}=A^{-1}\#_t B^{-1}\),
\[
    \lambda_{\min}(A\#_t B)
    \ge
    \lambda_{\min}(A)^{1-t}\lambda_{\min}(B)^t .
\]
Consequently,
\[
    \cond(A\#_tB)\le \cond(A)^{1-t}\cond(B)^t .
\]
Applying this with \(A=M_k\), \(B=S_k\), and \(t=\omega_k\), and using
\(\cond(M_k)\le\kappa_{\max}\) and \(\cond(S_k)\le\kappa_{\max}\), yields
\[
    \cond(M_{k+1})
    \le
    \cond(M_k)^{1-\omega_k}\cond(S_k)^{\omega_k}
    \le
    \kappa_{\max}.
\]
Therefore \(M_{k+1}\in\Cclass_k\).
\end{proof}

For convergence, only the properties summarized in
\cref{thm:metric-contract-rule} are needed.  \Cref{prop:spectral-output}
shows that the construction above satisfies them.

\begin{theorem}[Metric contract for the spectral construction]
\label{thm:metric-contract-rule}
Assume \(M_0\succ0\), \(\det(M_0)=1\), and \(\cond(M_0)\le\kappa_{\max}\).  Suppose
that the metric-selection routine returns \(M_{k+1}\in\Cclass_k\) for every \(k\).  Then
\begin{equation}
    \det(M_k)=1,
    \qquad
    \cond(M_k)\le\kappa_{\max},
    \qquad k=0,1,2,\ldots,
    \label{eq:det-cond-all-k}
\end{equation}
and
\begin{equation}
    \kappa_{\max}^{-(n-1)/n} I
    \preceq M_k\preceq
    \kappa_{\max}^{(n-1)/n} I.
    \label{eq:uniform-contract-kappa}
\end{equation}
Moreover,
\begin{equation}
    e^{-\delta_M/2}\normM{s}{M_k}
    \le \normM{s}{M_{k+1}}
    \le e^{\delta_M/2}\normM{s}{M_k}
    \qquad\forall s\in\R^n.
    \label{eq:one-step-distortion}
\end{equation}
\end{theorem}

\begin{proof}
The determinant and condition-number statements follow directly from
\(M_{k+1}\in\Cclass_k\) and the initialization.  Let \(0<\mu_1\le\cdots\le\mu_n\) be the
eigenvalues of any \(M\succ0\) with \(\det(M)=1\) and \(\cond(M)\le\kappa_{\max}\).
Since \(\mu_n/\mu_1\le\kappa_{\max}\) and \(\prod_i\mu_i=1\),
\begin{equation}
    1=\prod_{i=1}^n\mu_i
    \ge \mu_n\left(\frac{\mu_n}{\kappa_{\max}}\right)^{n-1}
    =\frac{\mu_n^n}{\kappa_{\max}^{n-1}},
\end{equation}
so \(\mu_n\le\kappa_{\max}^{(n-1)/n}\).  Similarly,
\begin{equation}
    1=\prod_{i=1}^n\mu_i
    \le \mu_1(\kappa_{\max}\mu_1)^{n-1},
\end{equation}
so \(\mu_1\ge\kappa_{\max}^{-(n-1)/n}\).  This proves
\eqref{eq:uniform-contract-kappa}.  Finally, the Loewner inequalities in
\eqref{eq:Ck-class} imply
\begin{equation}
    e^{-\delta_M}s^\top M_ks\le s^\top M_{k+1}s\le e^{\delta_M}s^\top M_ks.
\end{equation}
Taking square roots gives \eqref{eq:one-step-distortion}.
\end{proof}

\begin{corollary}[Nesting of consecutive metric balls]
\label{cor:metric-ball-nesting}
Under the assumptions of \cref{thm:metric-contract-rule}, for every \(x\in\R^n\) and
\(\Delta>0\),
\begin{equation}
    \Bball_{M_k}\bigl(x,e^{-\delta_M/2}\Delta\bigr)
    \subseteq
    \Bball_{M_{k+1}}(x,\Delta)
    \subseteq
    \Bball_{M_k}\bigl(x,e^{\delta_M/2}\Delta\bigr).
    \label{eq:metric-ball-nesting}
\end{equation}
\end{corollary}

\begin{proof}
The inclusions are exactly \eqref{eq:one-step-distortion} written in terms of metric
balls.
\end{proof}

\Cref{thm:metric-contract-rule,cor:metric-ball-nesting} provide the bridge from metric
design to convergence.  The construction can use curvature information when it is reliable,
while the induced coordinate systems remain uniformly bounded and change gradually.

\section{Convergence under the Metric Contract}
\label{sec:convergence}

We now verify that the abstract metric trust-region method is compatible with the
standard fully quadratic convergence mechanism.  The proof is stated in induced
variables, where the trust region is Euclidean.  The new ingredient is the uniform metric
contract: it transfers smoothness, model accuracy, and stationarity between the original
variables and the induced variables.  Once this transfer is made, the decrease and
radius-accounting arguments are the usual trust-region arguments.

\subsection{Model accuracy in induced variables}
\label{subsec:conv-assumptions}
The induced model \(\widetilde m_k\) is fully quadratic for \(\widetilde f_k\) on
\(\norm{y}\le\Delta_k\) if there exist constants \((\kappa_f^y,\kappa_g^y,\kappa_H^y)\)
such that
\begin{align}
    \abs{\widetilde f_k(y)-\widetilde m_k(y)} &\le \kappa_f^y\Delta_k^3,
    \label{eq:FQ-f}\\
    \norm{\nabla\widetilde f_k(y)-\nabla\widetilde m_k(y)} &\le \kappa_g^y\Delta_k^2,
    \label{eq:FQ-g}\\
    \norm{\nabla^2\widetilde f_k(y)-B_{y,k}}_2 &\le \kappa_H^y\Delta_k,
    \label{eq:FQ-H}
\end{align}
for all \(\norm{y}\le\Delta_k\).

\begin{assumption}[Regularity and model accuracy]
\label{asmp:regularity}
There is a region \(\mathcal R\subset\R^n\) containing all induced trust-region
balls
\[
    \{x_k+M_k^{-1/2}y:\ \|y\|\le \Delta_k\},
    \qquad k=0,1,\ldots,
\]
such that \(f\) is bounded below by \(f_\star\) on \(\mathcal R\), has an
\(L_2\)-Lipschitz Hessian on \(\mathcal R\), and satisfies
\[
    \sup_{x\in\mathcal R}\|\nabla^2 f(x)\|_2\le H_f<\infty .
\]
\Cref{asmp:metric-contract} holds.  The interpolation sets are uniformly
\(\Lambda\)-poised in the current metric ball, so that the induced models are
fully quadratic with constants independent of \(k\).
\end{assumption}

The stationarity measure used by the algorithm is the induced model-gradient norm
\(\chi_k=\norm{g_{y,k}}\).  The corresponding true stationarity measure is
\begin{equation}
    \chi_k^{\rm true}=\norm{\nabla f(x_k)}_{M_k^{-1}}
    =\norm{M_k^{-1/2}\nabla f(x_k)}.
    \label{eq:true-stationarity}
\end{equation}
Indeed, \(\nabla\widetilde f_k(0)=M_k^{-1/2}\nabla f(x_k)\), so
\(\chi_k^{\rm true}=\norm{\nabla\widetilde f_k(0)}\).

\begin{proposition}[Uniform fully quadratic constants]
\label{prop:FQ-transfer}
Assume \(f\) has an \(L_2\)-Lipschitz Hessian on the sampled region and that
\Cref{asmp:metric-contract} holds.  Then each induced function
\(\widetilde f_k(y)=f(x_k+M_k^{-1/2}y)\) has Lipschitz Hessian constant at most
\(L_2m_{\min}^{-3/2}\).  Consequently, if the interpolation sets are uniformly
\(\Lambda\)-poised in the induced balls, the fully quadratic constants in
\eqref{eq:FQ-f}--\eqref{eq:FQ-H} can be chosen independently of \(k\).
\end{proposition}

\begin{proof}
For any \(y,z\) in the induced ball, set
\(\xi_y=x_k+M_k^{-1/2}y\) and \(\xi_z=x_k+M_k^{-1/2}z\).  Then
\begin{align}
    \norm{\nabla^2\widetilde f_k(y)-\nabla^2\widetilde f_k(z)}_2
    &=\norm{M_k^{-1/2}\bigl(\nabla^2 f(\xi_y)-\nabla^2 f(\xi_z)\bigr)M_k^{-1/2}}_2 \notag\\
    &\le \norm{M_k^{-1/2}}_2^2 L_2\norm{\xi_y-\xi_z} \notag\\
    &\le L_2\norm{M_k^{-1/2}}_2^3\norm{y-z}
     \le L_2m_{\min}^{-3/2}\norm{y-z}.
    \label{eq:induced-Hess-Lip-proof}
\end{align}
The standard poisedness-to-fully-quadratic result for full quadratic interpolation then
gives uniform constants depending only on \(n\), \(\Lambda\), and the induced Lipschitz
constant.
\end{proof}

\begin{assumption}[Criticality safeguard]
\label{asmp:criticality}
There exists \(\mu_{\rm crit}>0\) with \(\kappa_g^y\mu_{\rm crit}^2\le 1/2\) such that,
whenever \(\chi_k\le1\), the algorithm enforces
\begin{equation}
    \Delta_k\le \mu_{\rm crit}\chi_k.
    \label{eq:criticality-safeguard}
\end{equation}
\end{assumption}

\begin{assumption}[Sufficient decrease and radius update]
\label{asmp:decrease-update}
There exist \(c_{\rm dec}>0\) and a uniform upper bound \(\bar h\) on \(\norm{B_{y,k}}_2\)
such that the trial step satisfies
\begin{equation}
    \pred_k\ge c_{\rm dec}
    \min\left\{\frac{\chi_k^2}{\bar h_+},\Delta_k\chi_k\right\},
    \qquad \bar h_+:=\max\{\bar h,1\}.
    \label{eq:cauchy-assumption}
\end{equation}
The trust-region update uses constants \(0<\eta_1\le\eta_2<1\) and
\(0<\gamma_{\rm dec}<1<\gamma_{\rm inc}\), and a maximum radius \(\Delta_{\max}\).
Rejected iterations shrink the radius by \(\gamma_{\rm dec}\); very successful iterations
with \(\rho_k\ge\eta_2\) increase the radius by at least \(\gamma_{\rm inc}\), unless the
maximum radius is active; all radii remain bounded by \(\Delta_{\max}\).
\end{assumption}

The uniform bound \(\bar h\) is consistent with \Cref{asmp:regularity}: for all
\(\|y\|\le\Delta_k\),
\[
    \|\nabla^2\widetilde f_k(y)\|_2
    \le
    \|M_k^{-1/2}\|_2^2 H_f
    \le
    m_{\min}^{-1}H_f .
\]
Together with the fully quadratic Hessian error bound, this gives for example
\[
    \|B_{y,k}\|_2
    \le
    m_{\min}^{-1}H_f+\kappa_H^y\Delta_{\max},
\]
so a finite uniform \(\bar h\) may be chosen.

\subsection{Acceptance at small radii}
\label{subsec:accuracy-acceptance}
The complexity proof uses model accuracy in two places: to compare model and true
stationarity, and to compare actual and predicted reduction.

\begin{lemma}[Stationarity transfer]
\label{lem:stationarity-bridge}
Under \Cref{asmp:regularity},
\begin{equation}
    \abs{\chi_k^{\rm true}-\chi_k}\le \kappa_g^y\Delta_k^2.
    \label{eq:stationarity-bridge}
\end{equation}
If additionally \Cref{asmp:criticality} holds and \(\chi_k\le1\), then
\begin{equation}
    \chi_k\ge \frac23\chi_k^{\rm true}.
    \label{eq:stationarity-bridge-two-thirds}
\end{equation}
\end{lemma}

\begin{proof}
The gradient bound \eqref{eq:FQ-g} at \(y=0\) gives
\begin{equation}
    \norm{M_k^{-1/2}\nabla f(x_k)-g_{y,k}}
    =\norm{\nabla\widetilde f_k(0)-\nabla\widetilde m_k(0)}
    \le \kappa_g^y\Delta_k^2.
\end{equation}
The reverse triangle inequality gives \eqref{eq:stationarity-bridge}.  If \(\chi_k\le1\),
then by the criticality safeguard,
\begin{equation}
    \kappa_g^y\Delta_k^2
    \le \kappa_g^y\mu_{\rm crit}^2\chi_k^2
    \le \frac12\chi_k^2
    \le \frac12\chi_k.
\end{equation}
Therefore \(\chi_k^{\rm true}\le \chi_k+\frac12\chi_k=\frac32\chi_k\), which is
\eqref{eq:stationarity-bridge-two-thirds}.
\end{proof}

\begin{lemma}[Actual--predicted reduction error]
\label{lem:ared-pred-gap}
Under \Cref{asmp:regularity},
\begin{equation}
    \abs{\ared_k-\pred_k}\le 2\kappa_f^y\Delta_k^3.
    \label{eq:ared-pred-gap}
\end{equation}
\end{lemma}

\begin{proof}
With \(s_k=T_k^{-1}y_k\), we have \(\widetilde f_k(0)=f(x_k)\) and
\(\widetilde f_k(y_k)=f(x_k+s_k)\).  Hence
\begin{equation}
    \ared_k-\pred_k
    =\bigl(\widetilde m_k(y_k)-\widetilde f_k(y_k)\bigr)
    -\bigl(\widetilde m_k(0)-\widetilde f_k(0)\bigr).
\end{equation}
Each term is bounded in absolute value by \(\kappa_f^y\Delta_k^3\) from
\eqref{eq:FQ-f}.
\end{proof}

\begin{lemma}[Acceptance below the critical radius]
\label{lem:accept-small-radius}
Assume \Cref{asmp:regularity,asmp:criticality,asmp:decrease-update}.  Fix
\(\varepsilon\in(0,1)\) and suppose \(\chi_k^{\rm true}>\varepsilon\).  Define
\begin{equation}
    \Delta_\varepsilon^{\rm vs}
    :=\min\left\{
    \left(\frac{(1-\eta_2)c_{\rm dec}\varepsilon}{3\kappa_f^y}\right)^{1/2},
    \frac{2\varepsilon}{3\bar h_+},
    \frac{\Delta_{\max}}{\gamma_{\rm inc}}
    \right\}.
    \label{eq:Delta-vs}
\end{equation}
If \(\Delta_k\le \Delta_\varepsilon^{\rm vs}\), then the iteration is very successful:
\(\rho_k\ge\eta_2\), and the radius is increased by at least \(\gamma_{\rm inc}\).
\end{lemma}

\begin{proof}
If \(\chi_k\le1\), then \cref{lem:stationarity-bridge} gives
\(\chi_k\ge(2/3)\chi_k^{\rm true}>2\varepsilon/3\).  If \(\chi_k>1\), the same lower bound
holds because \(\varepsilon<1\).  Thus \(\chi_k\ge2\varepsilon/3\) in all cases.  Since
\(\Delta_k\le2\varepsilon/(3\bar h_+)\le \chi_k/\bar h_+\), \Cref{asmp:decrease-update}
gives
\begin{equation}
    \pred_k\ge c_{\rm dec}\Delta_k\chi_k
    \ge \frac23 c_{\rm dec}\varepsilon\Delta_k.
\end{equation}
Together with \cref{lem:ared-pred-gap},
\begin{equation}
    \abs{\rho_k-1}
    \le \frac{2\kappa_f^y\Delta_k^3}{\pred_k}
    \le \frac{3\kappa_f^y\Delta_k^2}{c_{\rm dec}\varepsilon}
    \le 1-\eta_2.
\end{equation}
Hence \(\rho_k\ge\eta_2\).  The definition of \(\Delta_\varepsilon^{\rm vs}\) also ensures
\(\Delta_k\le\Delta_{\max}/\gamma_{\rm inc}\), so the update increases the radius by at
least \(\gamma_{\rm inc}\).
\end{proof}

\subsection{Evaluation complexity}
\label{subsec:complexity}
The final ingredient is the standard radius-accounting argument: below the
threshold \(\Delta_\varepsilon^{\rm vs}\), iterations are very successful and
push the radius upward, whereas large successful iterations decrease the
objective by a fixed multiple of \(\varepsilon^2\).  The following theorem
states the resulting bound for an arbitrary prefix of iterations on which
the target stationarity has not yet been reached.

\begin{theorem}[First-order evaluation complexity]
\label{thm:eval-complexity}
Assume \Cref{asmp:regularity,asmp:criticality,asmp:decrease-update}.  Let
\(\varepsilon\in(0,1)\), and suppose
\(\chi_k^{\rm true}>\varepsilon\) for \(k=0,1,\ldots,N-1\).  Then
\begin{equation}
    N=\cO(\varepsilon^{-2}).
    \label{eq:iter-complexity}
\end{equation}
Consequently the first hitting time
\(N_\varepsilon:=\min\{k:\chi_k^{\rm true}\le\varepsilon\}\) exists and satisfies the
same order bound.  If each geometry-repair event uses at most \(C_{\rm geom}q_n\) new
function evaluations for a constant independent of \(\varepsilon\), then the evaluation
complexity is
\begin{equation}
    \cO(q_n\varepsilon^{-2})=\cO(n^2\varepsilon^{-2}).
    \label{eq:eval-complexity}
\end{equation}
\end{theorem}

\begin{proof}
For every \(k<N\), the proof of \cref{lem:accept-small-radius} shows that
\(\chi_k\ge2\varepsilon/3\).  Consider a successful iteration with
\(\Delta_k\ge\Delta_\varepsilon^{\rm vs}\).  By the sufficient decrease condition,
\begin{equation}
    f(x_k)-f(x_{k+1})
    \ge \eta_1 c_{\rm dec}
    \min\left\{\frac{4\varepsilon^2}{9\bar h_+},
    \frac23\varepsilon\Delta_\varepsilon^{\rm vs}\right\}
    \ge c_s\varepsilon^2,
    \label{eq:large-success-decrease}
\end{equation}
where \(c_s>0\) is independent of \(\varepsilon\).  The last inequality uses
\(\Delta_\varepsilon^{\rm vs}\ge c_\Delta\varepsilon\) for a constant \(c_\Delta>0\),
which follows from \(0<\varepsilon<1\) because every term in
\(\Delta_\varepsilon^{\rm vs}\) is bounded below by a positive constant times
\(\varepsilon\).
Since \(f\) is bounded below, the number of such large-radius successful iterations is
\(\cO(\varepsilon^{-2})\).

By \cref{lem:accept-small-radius}, every iteration with
\(\Delta_k\le\Delta_\varepsilon^{\rm vs}\) is very successful and increases
the radius by at least \(\gamma_{\rm inc}\).  Rejected iterations shrink the
radius by \(\gamma_{\rm dec}\), and all radii remain bounded above by
\(\Delta_{\max}\).  A logarithmic radius-accounting argument then shows that
the number of unsuccessful iterations and small-radius successful iterations
is bounded by a constant multiple of the number of large-radius successful
iterations plus \(1+\abs{\log(\Delta_0/\Delta_\varepsilon^{\rm vs})}\).
Since this logarithmic term is
\(\cO(\log(1/\varepsilon))=\cO(\varepsilon^{-2})\) for
\(0<\varepsilon<1\), we obtain \(N=\cO(\varepsilon^{-2})\).

Each main iteration uses one trial evaluation.  Maintaining a fully quadratic model may
require geometry repair, and by assumption each repair event costs at most
\(C_{\rm geom}q_n\) new evaluations.  Hence the total evaluation count is
\(\cO(q_nN)=\cO(n^2\varepsilon^{-2})\).  Since no prefix longer than this bound can
satisfy \(\chi_k^{\rm true}>\varepsilon\) throughout, \(N_\varepsilon\) exists and obeys the
same bound.
\end{proof}

\begin{corollary}[Global first-order convergence]
\label{cor:global-conv}
Under the assumptions of \cref{thm:eval-complexity},
\begin{equation}
    \liminf_{k\to\infty}\chi_k^{\rm true}=0.
    \label{eq:global-conv}
\end{equation}
\end{corollary}

\begin{proof}
Fix any \(\varepsilon\in(0,1)\).  \Cref{thm:eval-complexity} implies that an iteration
with \(\chi_k^{\rm true}\le\varepsilon\) occurs after finitely many iterations.  Applying
this to a sequence \(\varepsilon_j\downarrow0\) gives a subsequence along which
\(\chi_k^{\rm true}\to0\).
\end{proof}

\begin{remark}[Dependence of the implicit constant on problem data]
\label{rem:constant-dependence}
The constant hidden in the \(\cO(\cdot)\) notation of \cref{thm:eval-complexity}
depends on
\(f(x_0)-f_{\rm low}\) (the initial optimality gap),
\(\bar h_+\) (the uniform upper bound on Hessian norms),
\(\kappa_{\max}\) (the metric condition-number cap),
\(\eta_1,\gamma_{\rm dec},\gamma_{\rm inc}\) (the trust-region management parameters),
and \(\Delta_{\max}\).  In particular, it does \emph{not} depend on the specific
sequence of metrics \(\{M_k\}\), because the spectral contract
(\cref{thm:metric-contract-rule}) confines all metrics to the same compact set.
The dependence on \(\kappa_{\max}\) enters through the Cauchy decrease constant
\(c_{\rm dec}\) in \cref{prop:metric-cauchy} and vanishes when the metric is the
identity, recovering the standard trust-region bound.
\end{remark}

\section{Numerical Experiments}
\label{sec:numerics}

The experiments examine the metric construction from three complementary angles.
The Mor\'e--Wild benchmark gives a standard comparison with established DFO solvers.
A controlled anisotropy study tests the central mechanism more directly by rotating and
rescaling benchmark problems.  Two-dimensional trajectories then show how the computed
ellipsoids behave along individual runs.  Together these experiments are meant to
illustrate the cost--accuracy trade-off of building a full quadratic curvature metric, not
to replace the deterministic convergence guarantees proved above.

\subsection{Experimental protocol}
\label{subsec:protocol}
All experiments use the Mor\'e--Wild budget convention: one simplex gradient
corresponds to \(n+1\) function evaluations \cite{MoreWild2009}.  A run is declared
solved at tolerance \(\tau\) if it reaches the Mor\'e--Wild threshold
\begin{equation}
    f(x)-f_\star \le \tau\bigl(f(x_0)-f_\star\bigr)
    \label{eq:solve-criterion}
\end{equation}
within the evaluation budget, with the implementation using the standard
\(f_\star\)-scaled absolute fallback when the initial gap is nonpositive.  Unless
stated otherwise, the budget is \(500(n+1)\) evaluations.  The reported final
relative error is
\[
    \frac{\abs{f(x)-f_\star}}{\max\{\abs{f_\star},\abs{f(x_0)-f_\star},1\}} .
\]
The parallel experiment harness uses per-task wall-clock cutoffs to avoid hung
runs: 300 seconds in the standard Mor\'e--Wild benchmark and 60 seconds in the
controlled-anisotropy study.  A run that exhausts either the evaluation budget or
the applicable wall-clock cutoff before reaching tolerance \(\tau\) is recorded
as a failure at \(\tau\).  These conventions apply uniformly to all solvers.
\MATRO{} uses the spectral metric
construction described in Section~\ref{subsec:spectral-rule}, with
\begin{gather}
    \eta_1=0.1,\quad \eta_2=0.5,\quad
    \gamma_{\rm dec}=0.5,\quad \gamma_{\rm inc}=2.5,
    \label{eq:experiment-params}\\
    \sigma=10^{-8},\quad \kappa_{\max}=10^6,\quad \delta_M=1.0.
    \notag
\end{gather}
All solvers use the same starting points.  Random seeds are shared across randomized
transformations and stochastic solvers; deterministic solvers are run on the same
transformed instances.

\subsection{Standard benchmark behavior}
\label{subsec:more-wild}
The first experiment uses the Mor\'e--Wild benchmark collection: 22 problem families,
37 instances, and dimensions \(n\in\{2,\ldots,12\}\).  We compare \MATRO{} with
four baselines: NEWUOA, UOBYQA, Nelder--Mead \cite{NelderMead1965}, and CMA-ES
\cite{HansenMullerKoumoutsakos2003}.  UOBYQA is the closest spherical
full-quadratic baseline because it uses the same order of interpolation points as
\MATRO{}.  NEWUOA is the lower-cost minimum-change baseline, using only
\(2n+1\) interpolation points.  Nelder--Mead and CMA-ES are included as classical
non-trust-region baselines.

\begin{table}[t]
\centering
\caption{Solve rates (\%) on the Mor\'e--Wild benchmark over 37 instances and five seeds.}
\label{tab:mw-solve-rates}
\begin{tabular}{lrrrr}
\toprule
Solver & \(\tau=10^{-1}\) & \(\tau=10^{-3}\) & \(\tau=10^{-5}\) & \(\tau=10^{-7}\) \\
\midrule
\MATRO{} & 86.5 & 83.8 & 83.8 & 75.7 \\
NEWUOA & 64.9 & 59.5 & 51.4 & 48.6 \\
UOBYQA & 62.2 & 62.2 & 51.4 & 48.6 \\
Nelder--Mead & 78.4 & 78.4 & 78.4 & 64.9 \\
CMA-ES & 86.5 & 86.5 & 76.2 & 66.5 \\
\bottomrule
\end{tabular}
\end{table}

\Cref{tab:mw-solve-rates} shows that \MATRO{} achieves the highest solve rates at
moderate-to-tight tolerances.  At \(\tau=10^{-1}\), \MATRO{} and CMA-ES both
reach 86.5\%, ahead of the classical trust-region solvers.  At
\(\tau=10^{-5}\), \MATRO{} leads with 83.8\%, followed by Nelder--Mead (78.4\%)
and CMA-ES (76.2\%).  At \(\tau=10^{-7}\), \MATRO{} solves 75.7\% of the runs,
compared with 66.5\% for CMA-ES and 64.9\% for Nelder--Mead.  NEWUOA and UOBYQA
show substantial timeout rates on the larger instances, so their solve rates should
be interpreted together with the wall-clock cutoff.  This pattern is consistent
with the intended role of the metric: curvature information becomes more informative
once the interpolation model is accurate enough to capture the local geometry.

\begin{figure}[t]
\centering
\begin{subfigure}[t]{0.48\textwidth}
\centering
\includegraphics[width=\textwidth]{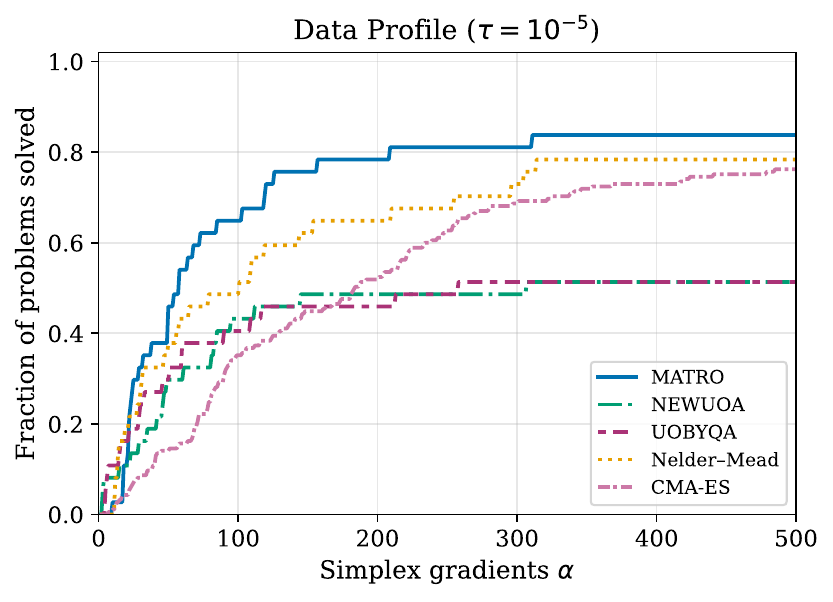}
\caption{\(\tau=10^{-5}\).}
\label{fig:mw-data-1e5}
\end{subfigure}\hfill
\begin{subfigure}[t]{0.48\textwidth}
\centering
\includegraphics[width=\textwidth]{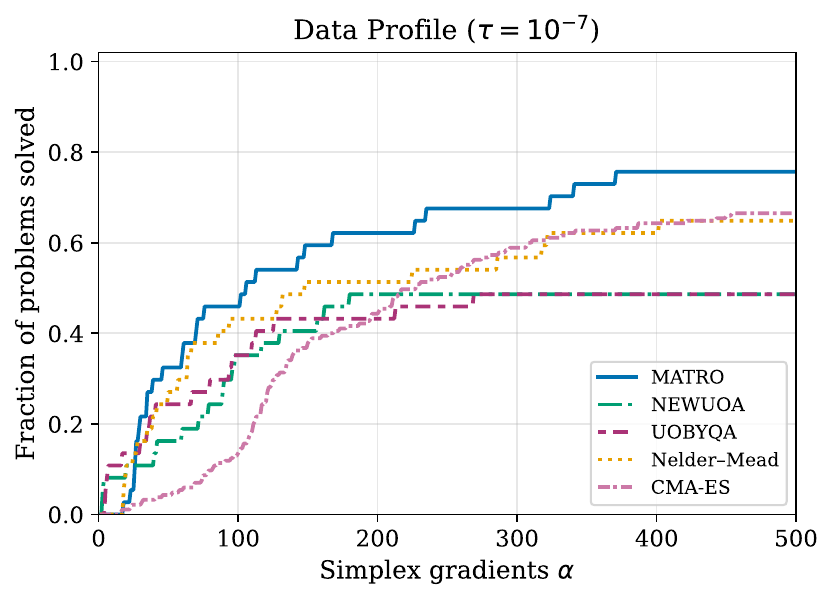}
\caption{\(\tau=10^{-7}\).}
\label{fig:mw-data-1e7}
\end{subfigure}
\caption{Data profiles on the Mor\'e--Wild benchmark.}
\label{fig:mw-data-profiles}
\end{figure}

\begin{figure}[t]
\centering
\begin{subfigure}[t]{0.48\textwidth}
\centering
\includegraphics[width=\textwidth]{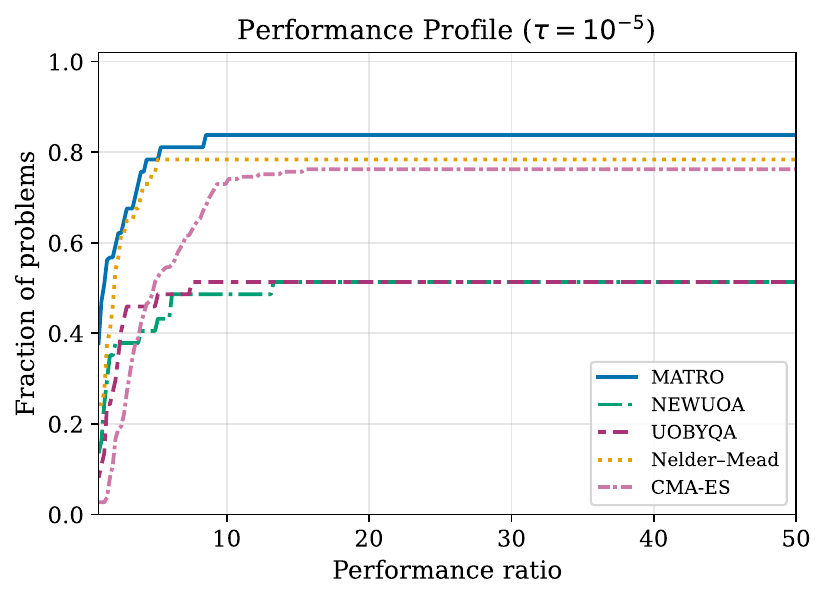}
\caption{\(\tau=10^{-5}\).}
\label{fig:mw-perf-1e5}
\end{subfigure}\hfill
\begin{subfigure}[t]{0.48\textwidth}
\centering
\includegraphics[width=\textwidth]{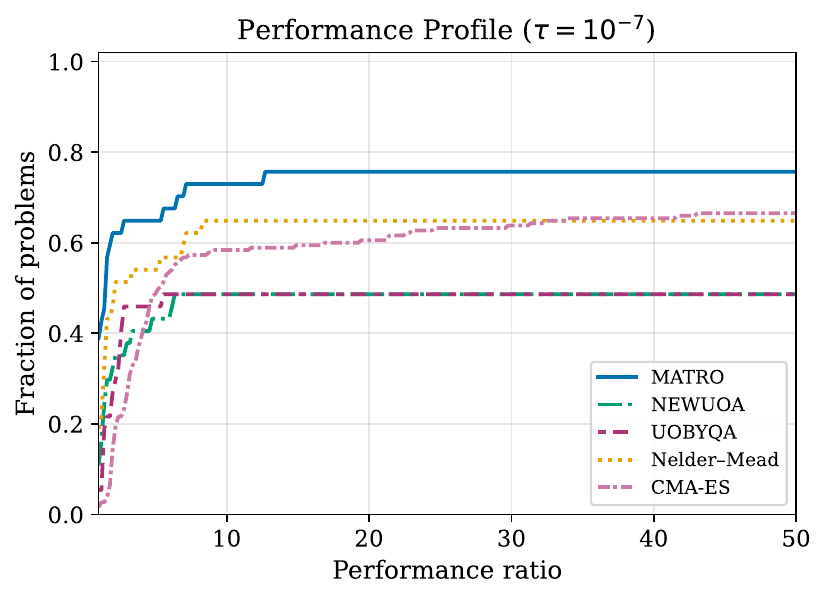}
\caption{\(\tau=10^{-7}\).}
\label{fig:mw-perf-1e7}
\end{subfigure}
\caption{Dolan--Mor\'e performance profiles on the Mor\'e--Wild benchmark.}
\label{fig:mw-perf-profiles}
\end{figure}

The data profiles in \cref{fig:mw-data-profiles} and the Dolan--Mor{\'e}
performance profiles in \cref{fig:mw-perf-profiles} \cite{DolanMore2002} give the
same trend at the profile level.  At small budgets, CMA-ES and Nelder--Mead
begin moving without first constructing a full quadratic interpolation model.  As
the budget and required accuracy increase, the full quadratic metric has more
opportunity to use curvature information, and \MATRO{} attains a larger solved
fraction.  The performance profiles also show that the relative ranking is
problem dependent, as expected for a method whose advantage depends on the
presence of stable local anisotropy.

The per-problem convergence curves support the aggregate profiles.  On classical
problems such as Rosenbrock and Powell, \MATRO{} converges comparably or faster than
the spherical baselines.  Its advantage is most visible on instances where NEWUOA and
UOBYQA exhaust the budget before reaching moderate accuracy, such as Watson
\((n=9)\), Chebyquad \((n=9)\), Osborne~1, and Osborne~2.

\begin{figure}[t]
\centering
\includegraphics[width=0.56\textwidth]{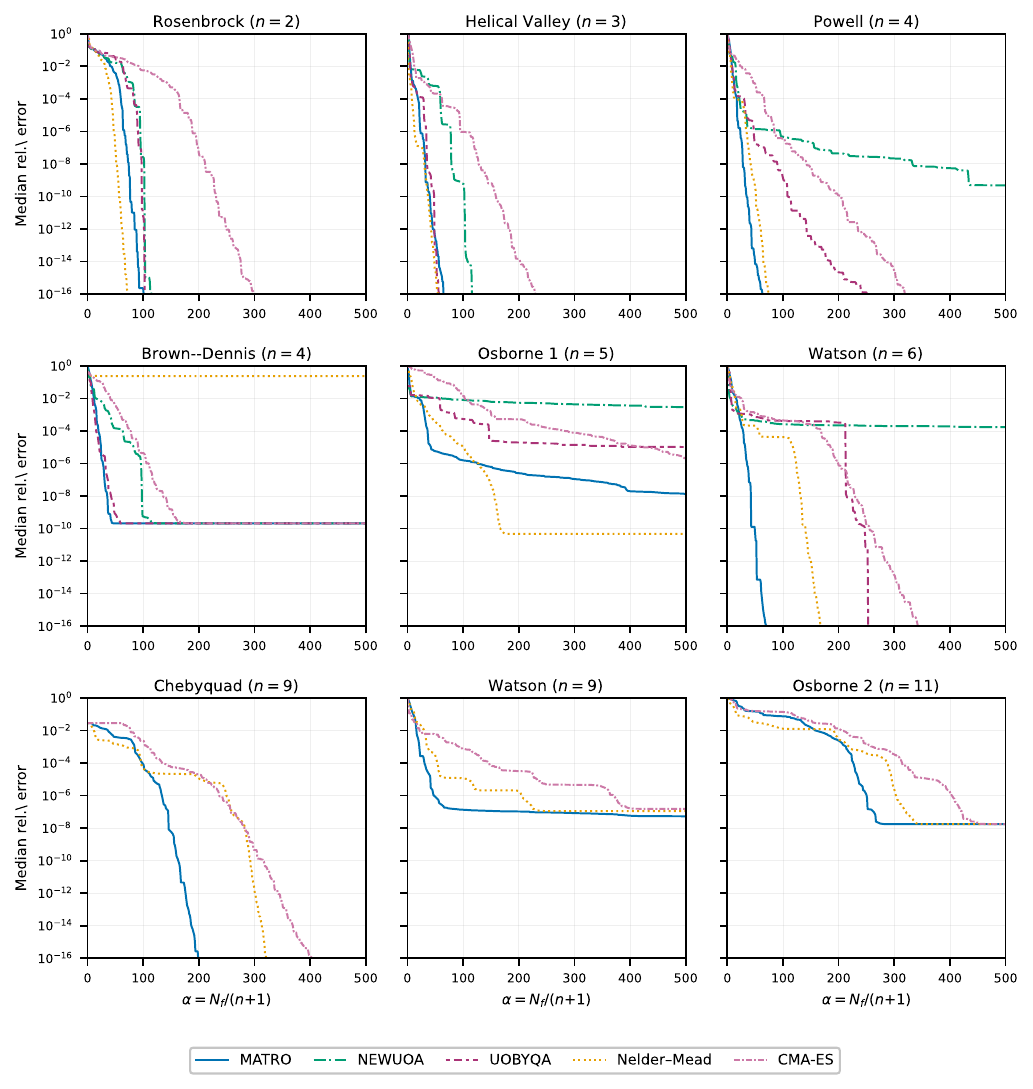}
\caption{Per-problem convergence on nine representative Mor\'e--Wild instances
(median over five seeds).}
\label{fig:conv-selected}
\end{figure}

\Cref{fig:conv-selected} illustrates these regimes; the full set of per-problem
curves is available with the plotting scripts in the repository.

\subsection{Response to controlled anisotropy}
\label{subsec:kappa-test}
The second experiment isolates the effect of anisotropy more directly.  We transform
each base problem by
\begin{equation}
    g(x)=f(Ax+b),
    \qquad
    A=Q\diag(\sigma_1,\ldots,\sigma_n)Q^\top,
    \qquad
    \cond(A)=\kappa,
    \label{eq:anisotropy-transform}
\end{equation}
where \(Q\) is a random orthogonal matrix and the singular values \(\sigma_i\) are
geometrically spaced between \(1\) and \(\kappa\).  At \(\kappa=1\), all \(\sigma_i\) are
equal and the transformation is orthogonal up to scaling.  The base problems are
Rosenbrock \((n=5)\), Powell \((n=4)\), Ellipsoid \((n=5)\), Wood \((n=4)\), and
Chebyquad \((n=6)\), each with
\(\kappa\in\{1,10,20,30,\ldots,100\}\).  For each \((\text{problem},\kappa)\)
setting, five random rotations and five solver seeds give 25 independent runs.
We compare \MATRO{} with NEWUOA.

\begin{figure}[t]
\centering
\includegraphics[width=0.92\textwidth]{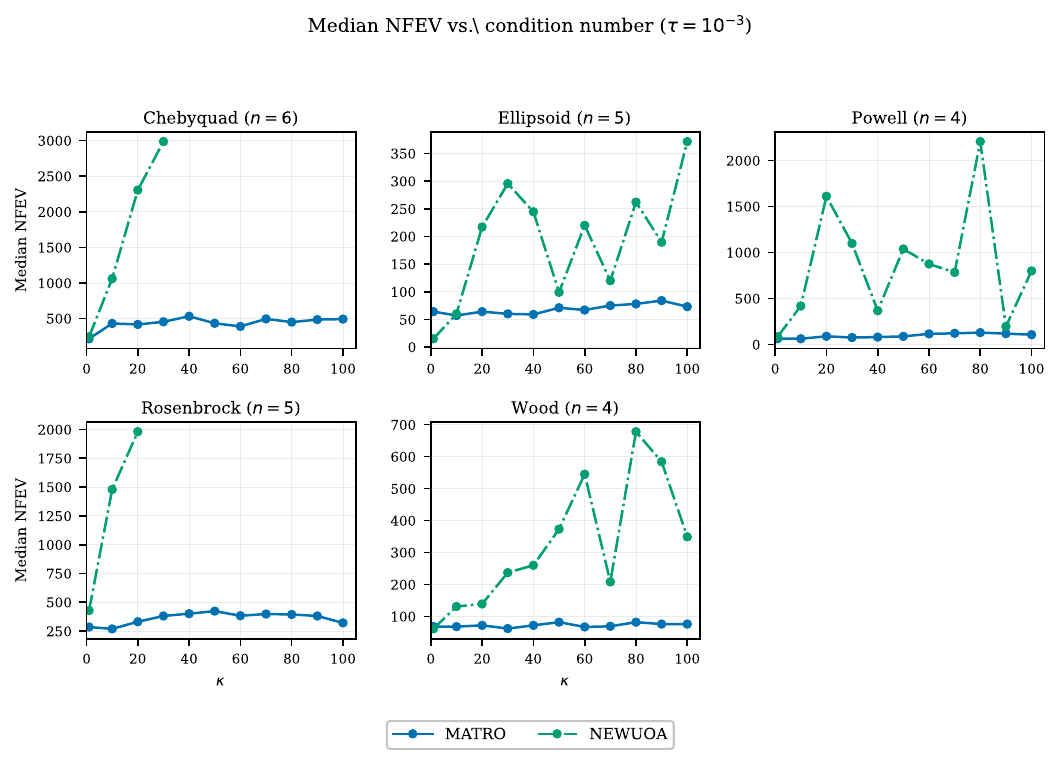}
\caption{Median NFEV versus condition number \(\kappa\) at \(\tau=10^{-3}\).
Each panel shows one base problem; medians are taken over solved runs among
25 independent instances per \(\kappa\).  A missing marker indicates that
no run in that group was solved.}
\label{fig:kappa-nfev}
\end{figure}
The median NFEV values are computed over solved runs only and should therefore
be interpreted together with the solve-rate table.

\begin{table}[t]
\centering
\caption{Solve rates (\%) on the controlled-anisotropy study at
\(\tau=10^{-3}\), over 25 runs per setting.  Here M denotes \MATRO{} and
N denotes NEWUOA; a zero entry means that no run was solved.}
\label{tab:kappa-solve-rates}
\small
\begin{tabular}{@{}r rr rr rr rr rr@{}}
\toprule
& \multicolumn{2}{c}{Chebyquad} & \multicolumn{2}{c}{Ellipsoid}
& \multicolumn{2}{c}{Powell} & \multicolumn{2}{c}{Rosenbrock}
& \multicolumn{2}{c}{Wood} \\
\cmidrule(lr){2-3}\cmidrule(lr){4-5}\cmidrule(lr){6-7}
\cmidrule(lr){8-9}\cmidrule(lr){10-11}
\(\kappa\) & M & N & M & N & M & N & M & N & M & N \\
\midrule
  1  & 100 & 100 & 100 & 100 & 100 & 100 & 100 &  40 & 100 & 100 \\
 10  & 100 & 100 & 100 &  92 & 100 & 100 & 100 &  60 & 100 & 100 \\
 20  & 100 &  60 & 100 &  96 & 100 & 100 & 100 &  60 & 100 & 100 \\
 30  & 100 &  20 & 100 &  88 & 100 &  80 & 100 &   0 & 100 & 100 \\
 40  & 100 &   0 & 100 &  88 & 100 &  40 & 100 &   0 & 100 & 100 \\
 50  & 100 &   0 & 100 &  68 & 100 &  40 & 100 &   0 & 100 & 100 \\
 60  & 100 &   0 & 100 &  80 & 100 &  60 & 100 &   0 & 100 & 100 \\
 70  & 100 &   0 & 100 &  76 & 100 &  40 & 100 &   0 & 100 & 100 \\
 80  & 100 &   0 & 100 &  68 & 100 &  60 & 100 &   0 & 100 & 100 \\
 90  & 100 &   0 & 100 &  64 & 100 &  20 & 100 &   0 & 100 & 100 \\
100  & 100 &   0 & 100 &  72 & 100 &  40 & 100 &   0 & 100 &  80 \\
\bottomrule
\end{tabular}
\end{table}

\Cref{fig:kappa-nfev,tab:kappa-solve-rates} show how increasing anisotropy
changes the relative behavior of the two solvers.  At \(\kappa=1\), both solvers
succeed on most problems; NEWUOA is faster on the well-conditioned Ellipsoid instance
because it uses a smaller interpolation set.  As \(\kappa\) grows, NEWUOA's solve
rate drops sharply on Chebyquad (reaching 0\% at \(\kappa\ge40\)), Rosenbrock
(0\% at \(\kappa\ge30\)), and Powell (below 50\% at \(\kappa\ge40\)), whereas
\MATRO{} maintains 100\% on all five problems across the entire range.  On
Ellipsoid, NEWUOA retains a partial solve rate but its median NFEV grows
more rapidly than \MATRO{}'s.  The \MATRO{} NFEV curves are nearly
flat: the metric absorbs the conditioning increase into the trust-region
shape, so that the subproblem in induced variables sees a much smaller
effective condition number.  The tighter-tolerance tables and raw data
are available in the accompanying repository.

\subsection{Trajectory-level geometry}
\label{subsec:two-d}
The third experiment visualizes the mechanism on two-dimensional problems.  The goal
is qualitative: the plots show how the metric changes the trust-region shape along a
trajectory and how this geometry relates to the observed error curves.

\begin{figure}[!t]
\centering
\begin{subfigure}[t]{0.39\textwidth}
\includegraphics[width=\textwidth]{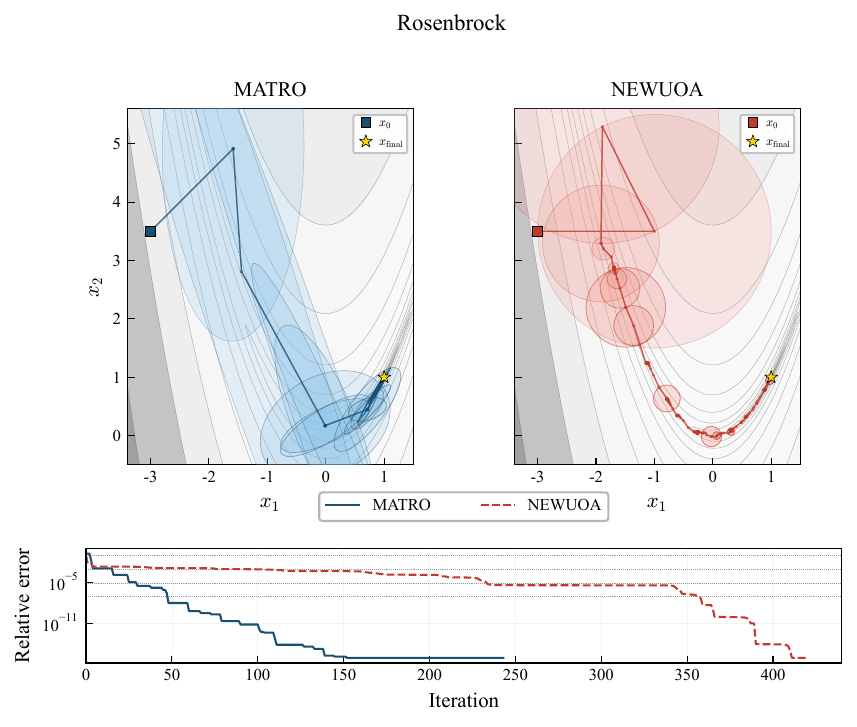}
\caption{Rosenbrock.}
\end{subfigure}\hfill
\begin{subfigure}[t]{0.39\textwidth}
\includegraphics[width=\textwidth]{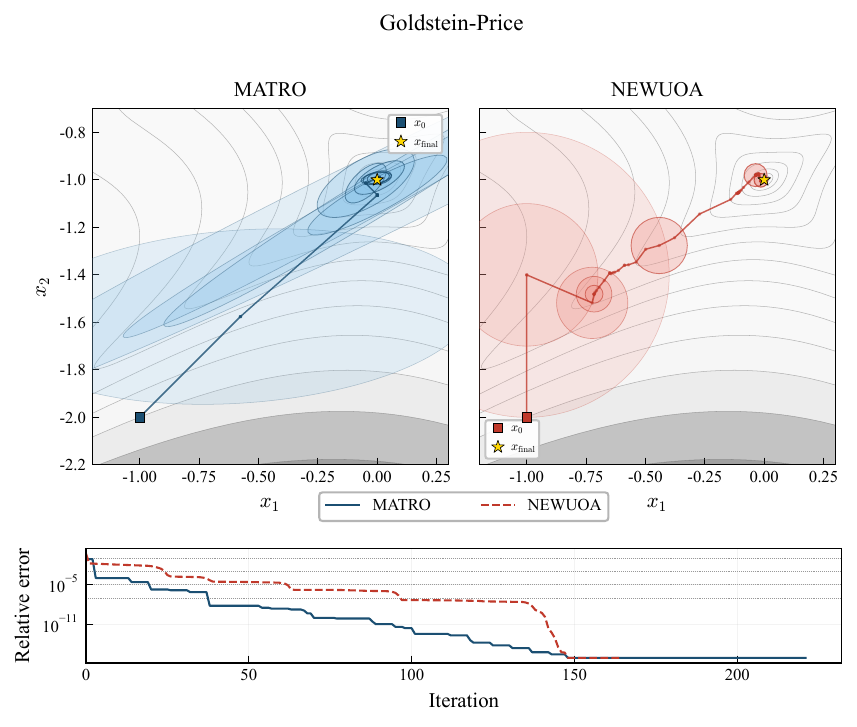}
\caption{Goldstein--Price.}
\end{subfigure}
\begin{subfigure}[t]{0.39\textwidth}
\includegraphics[width=\textwidth]{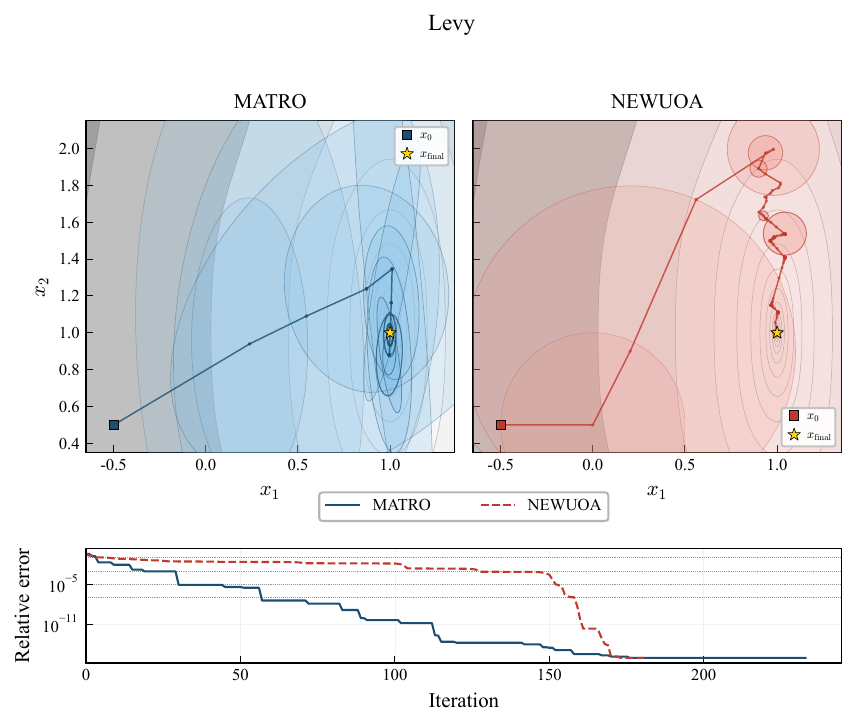}
\caption{L\'evy.}
\end{subfigure}\hfill
\begin{subfigure}[t]{0.39\textwidth}
\includegraphics[width=\textwidth]{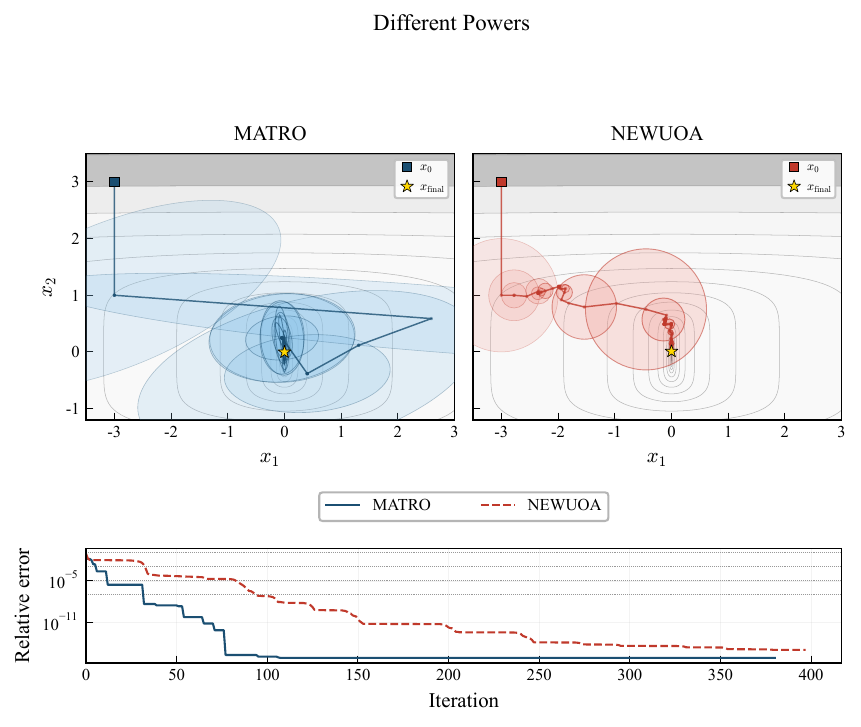}
\caption{Different Powers.}
\end{subfigure}
\caption{Two-dimensional trajectories.  \MATRO{} uses ellipsoidal regions and
NEWUOA uses circular regions.  The examples show valley alignment, metric
rotation, a nearly isotropic case, and changing curvature ratios.}
\label{fig:traj-pairs}
\end{figure}

\Cref{fig:traj-pairs} illustrates the mechanism at the trajectory level.  The
JSON data used to render these plots are included in the accompanying repository.

\subsection{Cost and observed limitations}
\label{subsec:overhead}
The cost of the curvature metric is concentrated in dense linear algebra.  Each \MATRO{}
iteration performs an \(\cO(n^3)\) eigendecomposition of \(H_{x,k}\) and stores an
\(\cO(n^2)\) metric matrix.  This cost is modest when function evaluations dominate,
but it can be visible on analytic benchmark functions where evaluations are cheap.  The
initialization cost is also larger than for NEWUOA: \MATRO{} requires
\(q_n=(n+1)(n+2)/2\) interpolation points before the first full quadratic model is
available.

The observed limitations are consistent with the role assigned to the metric.  If the
interpolation Hessian is inaccurate, the metric may align with numerical or sampling
artifacts.  If many eigenvalues are controlled by the spectral floor or cap, the metric may
underuse genuine curvature.  In strongly nonconvex regions, a positive definite
magnitude-based shape may not be the best search geometry.  These cases mark the
boundary between reliable local curvature information and curvature information that is
not yet accurate enough to define the trust-region shape.

\section{Conclusion}
\label{sec:conclusion}

This paper developed \MATRO{}, a fully quadratic derivative-free trust-region method
in which the trust region is an ellipsoid and the metric is selected from interpolation
curvature.  The analysis first treats the metric abstractly: after the induced-coordinate
map, the ellipsoidal subproblem is an ordinary Euclidean trust-region subproblem, and
classical poisedness and fully quadratic theory apply in those coordinates.  This shows
that changing the trust-region geometry need not change the deterministic
trust-region globalization mechanism, provided the metrics satisfy a uniform spectral
contract.

The metric-selection analysis then explains why a Hessian-derived metric is an appropriate
choice.  For positive definite quadratics, the determinant-normalized curvature metric
is the unique volume-normalized metric that makes the induced Hessian isotropic.  The
associated metric trust-region step is a truncated Newton step, and exact recovery occurs
once the Newton step is feasible.  For indefinite fitted Hessians, the absolute-curvature
metric balances the magnitudes of curvature directions while preserving their signs in
the induced model.  Fully quadratic Hessian accuracy gives a corresponding bound on
the induced conditioning of the true local Hessian.

Under the metric contract and the standard fully quadratic assumptions, the method
retains the usual first-order evaluation-complexity order
\(\cO(n^2\varepsilon^{-2})\) for full quadratic interpolation.  The numerical experiments
support the intended geometry-dependent interpretation.  The ellipsoidal region is most
helpful when the interpolation Hessian captures a stable anisotropic shape, especially in
rotated or poorly scaled local geometries.  Its cost is most visible at loose tolerances and
in problems where function evaluations are cheap relative to dense linear algebra.

Several directions remain open.  The most immediate is to reduce the cost of metric
selection through diagonal, block-diagonal, low-rank, or lazy-update variants.  A second
is to design more robust curvature-shape proxies in noisy or strongly nonconvex regimes.
A third is to combine metric selection with underdetermined or regression-based
quadratic models, so that the benefits of ellipsoidal trust regions are not tied to the full
\(\cO(n^2)\) interpolation cost.

Source code and scripts for the experiments are available at
\url{https://github.com/huwei0121/DFOETR}.



\paragraph{Reproducibility.}
The complete raw data, plotting scripts, per-problem convergence curves, the
controlled-anisotropy tables at the tighter tolerance, and the two-dimensional
trajectory JSON files are included in the accompanying repository.  The main
algorithmic parameters used to generate the figures are those listed in
\eqref{eq:experiment-params}; the full implementation configuration is recorded
with the experiment outputs.


\begin{thebibliography}{99}

\bibitem{AudetDennis2006MADS}
C.~Audet and J.~E. Dennis, Jr.,
\newblock Mesh adaptive direct search algorithms for constrained optimization,
\newblock {SIAM J. Optim.}, 17 (2006), pp.~188--217.

\bibitem{ConnGouldToint2000}
A.~R. Conn, N.~I.~M. Gould, and P.~L. Toint,
\newblock \emph{Trust-Region Methods},
\newblock MPS--SIAM Ser. Optim., SIAM, Philadelphia, PA, 2000.

\bibitem{ConnScheinbergVicente2008a}
A.~R. Conn, K.~Scheinberg, and L.~N. Vicente,
\newblock Geometry of interpolation sets in derivative free optimization,
\newblock {Math. Program.}, 111 (2008), pp.~141--172.

\bibitem{ConnScheinbergVicente2008b}
A.~R. Conn, K.~Scheinberg, and L.~N. Vicente,
\newblock Geometry of sample sets in derivative-free optimization: polynomial regression and underdetermined interpolation,
\newblock {IMA J. Numer. Anal.}, 28 (2008), pp.~721--748.

\bibitem{ConnScheinbergVicente2009}
A.~R. Conn, K.~Scheinberg, and L.~N. Vicente,
\newblock \emph{Introduction to Derivative-Free Optimization},
\newblock MPS--SIAM Ser. Optim., SIAM, Philadelphia, PA, 2009.

\bibitem{DolanMore2002}
E.~D. Dolan and J.~J. Mor\'e,
\newblock Benchmarking optimization software with performance profiles,
\newblock {Math. Program.}, 91 (2002), pp.~201--213.

\bibitem{HansenMullerKoumoutsakos2003}
N.~Hansen, S.~D. M\"uller, and P.~Koumoutsakos,
\newblock Reducing the time complexity of the derandomized evolution strategy with covariance matrix adaptation (CMA-ES),
\newblock {Evol. Comput.}, 11 (2003), pp.~1--18.

\bibitem{LarsonMenickellyWild2019}
J.~Larson, M.~Menickelly, and S.~M. Wild,
\newblock Derivative-free optimization methods,
\newblock {Acta Numer.}, 28 (2019), pp.~287--404.

\bibitem{MoreWild2009}
J.~J. Mor\'e and S.~M. Wild,
\newblock Benchmarking derivative-free optimization algorithms,
\newblock {SIAM J. Optim.}, 20 (2009), pp.~172--191.

\bibitem{NelderMead1965}
J.~A. Nelder and R.~Mead,
\newblock A simplex method for function minimization,
\newblock {Comput. J.}, 7 (1965), pp.~308--313.

\bibitem{Powell2002UOBYQA}
M.~J.~D. Powell,
\newblock {UOBYQA}: Unconstrained optimization by quadratic approximation,
\newblock {Math. Program.}, 92 (2002), pp.~555--582.

\bibitem{Powell2006NEWUOA}
M.~J.~D. Powell,
\newblock The {NEWUOA} software for unconstrained optimization without derivatives,
\newblock in \emph{Large-Scale Nonlinear Optimization}, G.~Di Pillo and M.~Roma, eds.,
Nonconvex Optim. Appl. 83, Springer, New York, 2006, pp.~255--297.

\bibitem{Powell2009BOBYQA}
M.~J.~D. Powell,
\newblock The {BOBYQA} algorithm for bound constrained optimization without derivatives,
\newblock Tech. Rep. DAMTP 2009/NA06, Department of Applied Mathematics and Theoretical Physics,
University of Cambridge, Cambridge, UK, 2009.

\bibitem{WildShoemaker2011}
S.~M. Wild and C.~A. Shoemaker,
\newblock Global convergence of radial basis function trust-region derivative-free algorithms,
\newblock {SIAM J. Optim.}, 21 (2011), pp.~761--781.

\bibitem{XieYuan2025SIAM}
P.~Xie and Y. Yuan,
\newblock A derivative-free method using a new underdetermined quadratic interpolation model,
\newblock {SIAM J. Optim.}, 35 (2025), pp.~1110--1133.

\bibitem{XieYuan2026H2}
P.~Xie and Y. Yuan,
\newblock Least \(H^2\)-norm updating of quadratic interpolation models for derivative-free trust-region algorithms,
\newblock {IMA J. Numer. Anal.}, 46 (2026), pp.~21--50.


\bibitem{LiZhouXieLi2025}
L. Li, Y. Zhou, P. Xie, and H. Li,
``A spectral Levenberg--Marquardt--deflation method for multiple solutions of semilinear elliptic systems,''
\emph{Journal of Computational and Applied Mathematics},
vol. 475, article 116998, 2025.

\bibitem{YeLiXieYu2025}
Y. Ye, L. Li, P. Xie, and H. Yu,
``An improved adaptive orthogonal basis deflation method for multiple solutions with applications to nonlinear elliptic equations in varying domains,''
\emph{Journal of Computational Mathematics},
vol. 44, no. 3, pp. 794--818, 2025.

\bibitem{XieYuan2023}
P. Xie and Y.-x. Yuan,
``A derivative-free optimization algorithm combining line-search and trust-region techniques,''
\emph{Chinese Annals of Mathematics, Series B},
vol. 44, no. 5, pp. 693--708, 2023.


\bibitem{XieWild2026ReMU}
P.~Xie and S.~M. Wild,
\newblock {ReMU}: regional minimal updating for model-based derivative-free optimization,
\newblock {Optim. Methods Softw.}, published online, 2026.

\bibitem{XieYuan2025DFOTO}
P.~Xie and Y. Yuan,
\newblock Derivative-free optimization with transformed objective functions and the algorithm based on the least Frobenius norm updating quadratic model,
\newblock {J. Oper. Res. Soc. China}, 13 (2025), pp.~327--363.

\bibitem{XieYuan2026MoSub}
P.~Xie and Y. Yuan,
\newblock A new two-dimensional model-based subspace method for large-scale unconstrained derivative-free optimization: 2D-MoSub,
\newblock {Optim. Methods Softw.}, 41 (2026), pp.~118--150.

\bibitem{Zhang2023PRIMA}
Z.~Zhang,
\newblock {PRIMA}: Reference implementation for Powell's methods with modernization and amelioration,
\newblock preprint, arXiv:2307.12962, 2023.

\end{thebibliography}
\end{document}